\documentclass[12pt,leqno,draft]{article} 
\usepackage{amsmath,amssymb,amsthm,amsfonts,bm}
\usepackage{enumerate}
\usepackage{indentfirst}
\usepackage{cite}
\usepackage[usenames,dvipsnames]{color}
\definecolor{viola}{rgb}{0.3,0,0.7}
\definecolor{ciclamino}{rgb}{0.5,0,0.5}

\def\pier #1{{\color{viola}#1}}
\def\pier #1{#1}

\makeatletter
\def\@cite#1#2{[{{\bfseries #1}\if@tempswa , #2\fi}]}
\renewcommand{\section}{%
\@startsection{section}{1}{\z@}
{0.5truecm plus -1ex minus -.2ex}%
{1.0ex plus .2ex}{\bfseries\large}}
\def\@seccntformat#1{\csname the#1\endcsname.\ }
\makeatother

\numberwithin{equation}{section} 
\newtheorem{thm}{Theorem}[section]

\newtheorem{lem}[thm]{Lemma}

\theoremstyle{definition}
\newtheorem{df}{Definition}[section]
\newtheorem{remark}{Remark}[section]
\newtheorem*{prth2.1}{Proof of Theorem 2.1}
\newtheorem*{prth2.2}{Proof of Theorem 2.2}
\newtheorem*{prlem3.1}{Proof of Lemma 3.1}
\newtheorem*{prlem3.2}{Proof of Lemma 3.2}

\newcommand{\ep}{\varepsilon}

\newcommand{\Rd}{\mathbb{R}^d}

\def\checkmmode #1{\relax\ifmmode\hbox{#1}\else{#1}\fi}

\let\hat\widehat
\def\Beta{\hat{\vphantom t\smash\beta\mskip2mu}\mskip-1mu}
\def\Pi{\hat\pi}

\begin{document}
%==========================title==========================
\begin{center}
    \Large{{\bf Global existence for a phase separation system \\ 
                     deduced from the entropy balance
           }}
\end{center}
\vspace{5pt}
%===========================author=========================
\begin{center}
    Pierluigi Colli\\
    \vspace{2pt}
    Dipartimento di Matematica ``F. Casorati", Universit\`a di Pavia \\ 
    and Research Associate at the IMATI -- C.N.R. Pavia \\ 
    via Ferrata 5, 27100 Pavia, Italy\\
    {\tt pierluigi.colli@unipv.it}\\
    \vspace{12pt}
    Shunsuke Kurima%
   \footnote{Corresponding author}\\
    \vspace{2pt}
    Department of Mathematics, 
    Tokyo University of Science\\
    1-3, Kagurazaka, Shinjuku-ku, Tokyo 162-8601, Japan\\
    {\tt shunsuke.kurima@gmail.com}\\
    \vspace{2pt}
\end{center}
\begin{center}    
    \small \today
\end{center}

\vspace{2pt}
%=====================  Abstract  =======================
\newenvironment{summary}
{\vspace{.5\baselineskip}\begin{list}{}{%
     \setlength{\baselineskip}{0.85\baselineskip}
     \setlength{\topsep}{0pt}
     \setlength{\leftmargin}{12mm}
     \setlength{\rightmargin}{12mm}
     \setlength{\listparindent}{0mm}
     \setlength{\itemindent}{\listparindent}
     \setlength{\parsep}{0pt}
     \item\relax}}{\end{list}\vspace{.5\baselineskip}}
\begin{summary}
{\footnotesize {\bf Abstract.}
     This paper is concerned with a 
thermomechanical model describing phase separation \pier{phenomena} in
terms of the entropy balance and equilibrium equations for the  
microforces. 
The related system is highly nonlinear and admits singular potentials 
in the phase equation. 
Both the viscous and the non-viscous cases are considered 
in the Cahn--Hilliard relations characterizing the phase dynamics.  
The entropy balance is written in terms of the absolute temperature 
and of its logarithm, appearing under time derivative. 
The initial and boundary value problem is considered  
for the system of partial differential equations.   
The existence of a global solution is proved via some approximations 
involving Yosida regularizations and a suitable time discretization.   
}
\end{summary}
\vspace{10pt}

\newpage
%%==============================================================%%
%%==============                                  ==============%%
%%======                      Section1                    ======%%
%%====                                                      ====%%
%%==                                                         ==%%
%%====               Introduction Results            ====%%
%%======                                                  ======%%
%%==============                                  ==============%%
%%==============================================================%%

\section{Introduction and results} \label{Sec1}

In this paper we address the following system of partial differential equations
\begin{align}
         &\partial_t (c_{s}\ln\theta + \lambda(\varphi)) 
         - \eta\Delta\theta = f  , 
         %\quad \mbox{in}\ \Omega\times(0, T), 
 \label{1pier}
 \\
        &\partial_{t}\varphi - \Delta \mu = 0 ,
         %\quad \mbox{in}\ \Omega\times(0, T), 
 \label{2pier}
 \\
         &\mu = \tau\partial_{t}\varphi 
         - \gamma\Delta\varphi + \xi + \sigma'(\varphi) 
         - \lambda'(\varphi)\theta,\quad \xi \in \beta(\varphi) ,  
         % \quad \mbox{in}\ \Omega\times(0, T), 
\label{3pier} 
\end{align}
in the cylindrical domain $\Omega \times (0,T)$, where $\Omega\subset \mathbb{R}^3$ 
is a bounded and smooth open set, and $T>0$ denotes some final time. The system is complemented by the boundary conditions 
\begin{equation}
         \eta\partial_{\nu}\theta + \alpha_{\Gamma}(\theta-\theta_{\Gamma}) = 0, \quad
         \partial_{\nu}\mu = 0 , \quad \partial_{\nu}\varphi = 0    \quad                                
          \mbox{ on}\ \Gamma\times(0, T),
\label{4pier}
\end{equation}
where $\Gamma$ denotes the smooth boundary of $\Omega$, and by the initial conditions
\begin{equation}
        (\ln\theta)(0) = \ln\theta_0,\quad \varphi(0)=\varphi_0        \quad                                   
         \mbox{ in}\ \Omega.  
\label{5pier}         
\end{equation}
The equations and conditions \eqref{1pier}-\eqref{5pier}, with an inclusion in 
\eqref{3pier} as well, yield an initial and boundary value problem for the nonlinear 
phase field system \eqref{1pier}-\eqref{3pier}, which results from a thermomechanical 
model describing phase separation in terms of the variables {\it absolute temperature} $
\theta$, {\it order parameter} $\varphi$ and {\it chemical potential} $\mu$ 
(cf.~\cite{Bon2005, bcfgdue, BCF, BF, BFR, MiSc}).
Equation \eqref{1pier} gives account of an entropy balance and 
\eqref{2pier},~\eqref{3pier} render the equilibrium equations for the  
microforces that govern the phase separation phenomenon. Note that combining \eqref{2pier} and \eqref{3pier} yields actually the well-known Cahn--Hilliard equation 
in which the mixed term $- \lambda'(\varphi)\theta$ accounts for the contribution of temperature. Concerning equation \eqref{1pier}, let us emphasize that this equation is singular with respect to the temperature, 
due to the presence of the logarithm, which forces the temperature
to assume only positive values (in accordance with physical consistency; 
similar systems have been studied in the literature, e.g.,  \cite{BBR,bcfgdue,BCFG2007,BCFG2009,BFR,BCF,BCG,BR2007,CC,C2018,GR2006}).

Here, the positive 
constant $c_s$ represents the specific heat of the system; $\eta >0$ is a
thermal parameter for the entropy flux;  the factor $\lambda'(\varphi)$ in
$\partial_t (\lambda  (\varphi )) = \lambda'(\varphi)\partial_{t} \varphi$ plays as
latent heat, and the known right-hand side $f$ stands for an external entropy source. 
Moreover, $\gamma >0$ is a small positive parameter and the coefficient $\tau $ can be positive or zero: accordingly, we speak of viscous Cahn--Hilliard or non-viscous 
Cahn--Hilliard system, respectively. In fact, the term $\tau\partial_{t}\varphi$ 
represents a viscosity term in the description of the order parameter dynamics. Concerning the nonlinearities, we inform that $\lambda $ and $\sigma$ are two smooth functions in $\mathbb{R}$, with at most quadratic growth at infinity since the derivatives $\lambda'$ and $\sigma'$ are Lipschitz continuous in $\mathbb{R}$. On the other hand, the nonlinearity $\beta$ may represent a maximal monotone graph in  $\mathbb{R} \times \mathbb{R}$, possibly multivalued, with $0\in \beta (0)$, which turns out to be the subdifferential of a convex and lower semicontinuous function  $\hat{\beta} : \mathbb{R} \to [0,+\infty] $ with minimum value $0$ assumed in $0$.

Typical and physically significant examples for $\,\hat{\beta}\,$ 
are the {\em regular potential}, the {\em logarithmic potential\/},
and the {\em indicator potential\/}, which are given, in this order,~by
\begin{align}
  & \hat{\beta}_{reg}(r) := \frac 14 \, r^4 \,,
  \quad r \in \mathbb{R}, 
  \label{regpot}
  \\[3mm]
  & \hat{\beta}_{log}(r) :=  
  \begin{cases}
  (1+r)\ln (1+r)+(1-r)\ln (1-r)  \ &\hbox{if } r \in (-1,1),\\[2mm]
  2\ln 2 \quad &\hbox{if }  r \in \{-1,1\},\\[2mm]
  + \infty  \quad &\hbox{if } r \in (-\infty,-1) \cup (1, +\infty ),
  \end{cases} 
  \label{logpot}
  \\[3mm]
  & \hat{\beta}_{ind}(r) := 
  \begin{cases} 0 
  \quad &\hbox{if $|r|\leq1$},
  \\[2mm]
  +\infty
  \  &\hbox{if $|r|>1$}.
  \end{cases}
  \label{obspot}
\end{align}
Note that in cases like \eqref{regpot} the subdifferential $\beta$ coincides with 
the derivative $\Beta'(r) = r^3 $ and \eqref{3pier} becomes an equation with $\xi=\varphi^3$. Almost the same occurs in the case \eqref{logpot} since the subdifferential 
$\beta$ has the domain $(-1,1)$ and, in its domain, $\beta (r) = \ln (1+r) - \ln (1-r) $. 
On the other hand, in the case \eqref{obspot} $\beta$ is actually a graph  
\begin{equation*}
s\in \beta (r) \quad \hbox{ if and only if } \quad r\in [-1,1], 
\quad  s 
\begin{cases}
\leq 0 \quad  \hbox{if $r=-1$},\\
= 0 \quad  \hbox{if $-1<r<1$},\\
\geq 0 \quad  \hbox{if $r=1$}. 
\end{cases} 
\end{equation*}
Next, after the presentation of possible functions $\Beta$, let us remark that the sum 
of the three terms
\begin{equation}
\label{6pier} 
 \hat{\beta} (\varphi) + \sigma (\varphi ) - \lambda(\varphi)\theta
\end{equation} 
constitutes a part of the (local) free energy density, which has usually the structure of a double-well or multi-well potential and, according to different values of
the temperature $\theta$, may prefer one or another of the possible minimal states, or become fully convex if $\theta$ enters a suitable range of temperatures.  For instance, one can take   
$$
  \sigma (r ) = \theta_a\,r  - \theta_b\, r^2 ,  \quad \ \lambda(r)= r - r^2, 
  \quad r \in \mathbb{R},
  $$
with $\theta_a < \theta_b$ denoting some critical temperatures, so that if $\theta < \theta_b$ one of the two minima if preferred according whether or not $\theta > \theta_a$; instead, if $\theta > \theta_b$ the potential in \eqref{6pier} is convex. 
 
About the 
boundary conditions \eqref{4pier}, we point out that the boundary condition for $\theta$ states that the external flow on the boundary is proportional to the difference of temperatures between the interior and exterior of the body, via the given positive function $\alpha_\Gamma$ on  $\Gamma$, where the external temperature $\theta_{\Gamma}$ is prescribed on $\Gamma \times (0,T)$. On the other hand, in \eqref{4pier} two no-flux boundary conditions are assumed for $\mu$ and $\varphi$, as usual for the dynamics of Cahn--Hilliard models. In particular, this entails that the phase field system under consideration is of {\it conserved} type, since the integration by parts over $\Omega \times (0,t) $ of the equation~\eqref{2pier} yields the conservation property for the mean value of $\varphi$, i.e.,
\begin{equation*}
\frac1{|\Omega|} \int_\Omega \varphi (t) = \frac1{|\Omega|} \int_\Omega \varphi_0 \quad \hbox{for }\, t\in (0,T), 
\end{equation*}
where $\varphi_0$ represents the known initial value for the order parameter in 
\eqref{5pier} (similar phase field systems of conserved type are studied in 
\cite{CGGS1, CGGS2, CGLN}. 
Note that $\theta_0$ stands for the initial value of the temperature,
but the right initial condition to prescribe is for $ \ln\theta $, since it is this function $ \ln\theta $ appearing under the time derivative in \eqref{1pier}.

The related system for phase transitions, which gives rise to a phase field model of nonconserved type, has been intensively discussed in the papers \cite{bcfgdue, BCFG2007}
by taking into account some memory effects as well. Indeed, in the case of a phase transition, the equations \eqref{2pier} and \eqref{3pier} are replaced by a single equation of Allen--Cahn type:
$$\partial_{t}\varphi 
         - \gamma\Delta\varphi + \xi + \sigma'(\varphi) 
         - \lambda'(\varphi)\theta,\quad \xi \in \beta(\varphi) ,  
         % \quad \mbox{in}\ \Omega\times(0, T), 
$$ 
and the chemical potential does not play any role. The approach of \cite{bcfgdue, 
BCFG2007} follows some ideas previously developed in \cite{BCF}, with the aim of 
combining the thermal memory theory by Gurtin-Pipkin~\cite{GPmodel} with additional 
dissipative instantaneous contributions coming from a pseudo-potential of dissipation.
The use of an entropy balance is recovered 
from a rescaling (with respect to the absolute temperature) of the energy balance, 
under the small perturbations assumption (see, in particular, \cite{bcfgdue}).
In \cite{BFR} a fairly general theory is introduced, in which 
a dual approach (mainly in the sense of convex analysis) is considered, and 
the entropy and the history of the entropy flux are taken as state variables, 
along with the phase parameter and possibly its gradient. 
Then the dissipative functional is written in terms of a dissipative contribution in the entropy flux and for the time derivative of the phase parameter. This argumention may be 
understood in the light of general theories discussed in \cite{FaGioMo, Frel, Gu, Muller} as well. However, we have to point 
out that this framework is not far from the approach by Green-Naghdi 
\cite{GreenNaghdi} (see also \cite{M2017}) and Podio-Guidugli~\cite{Podio1}, 
in which some 
{\it thermal displacement} is introduced as state variable and 
the equations come from a  generalization of the principle of 
virtual powers, in which thermal forces are included. %(see also Remark \ref{confronto}). 
As a consequence, in this setting, the entropy equation is  formally obtained as a momentum balance (i.e., a balance of thermal forces acting in the system). Let us mention the related contributions~\cite{CC1, CC2}, where some asymptotic analyses are carried out to find the interconnections among peculiar Green and Naghdi types, and \cite{CGMQ}, where a model with two temperatures for heat conduction with memory, apt to describe transition phenomena in nonsimple materials, is investigated. 

Another example of an entropy balance equation can be found in the recent paper~\cite{KS}, where
a diffuse interface model is proposed to describe the multi-component two-phase fluid flow with  partial miscibility, by combining the first law of thermodynamics and related thermodynamical relations.

Eventually, let us quote the paper \cite{GR2007} and compare the results with ours. 
Indeed, the system \eqref{1pier}-\eqref{3pier} was already considered in the paper 
\cite{GR2007}, by dealing with Dirichlet boundary conditions for the temperature, 
and well-posedness results were discussed along with the investigation of the $\omega$-limit set for the system. We advice the reader that the existence results contained in \cite{GR2007} are similar to ours, although in the present contribution we are able to improve the thesis in the interesting non-viscous case $\tau=0$, by allowing a (significant) quadratic growth for $\lambda$, while \cite{GR2007}
only deals with Lipschitz continuous functions $\lambda$ in this limiting case. Moreover, we can give a complete proof of the existence of solutions, with respect to \cite{GR2007} 
where a priori estimates are plainly derived on the direct problem without implementing a suitable approximation. Furthermore, we treat the case of the third-type 
boundary condition for $\theta$ (cf.~\eqref{4pier}), differently from Dirichlet boundary conditions used in \cite{GR2007} (and already examined in~\cite{BCFG2007} for the nonconserved system). Thus, we ask the readers to follow our arguments and refer to the next section, where the problems are mathematically stated and a precise formulation is given with assumptions and results.

\section{\pier{Statement of problems} and main results}\label{SecPier}

In this paper we consider 
the following initial-boundary value problems   
%-----------------------------------------------------------------
%
%               Introduction of (P)
%
%-----------------------------------------------------------------
 \begin{equation*}\tag*{(P)}\label{P}
     \begin{cases}
         \partial_t (c_{s}\ln\theta + \lambda(\varphi)) - \eta\Delta\theta = f   
         & \mbox{in}\ \Omega\times(0, T), 
 \\[2mm]
        \partial_{t}\varphi - \Delta \mu = 0 
         & \mbox{in}\ \Omega\times(0, T), 
 \\[2mm]
         \mu = - \gamma\Delta\varphi + \xi + \sigma'(\varphi) 
         - \lambda'(\varphi)\theta,\ \xi \in \beta(\varphi)   
         & \mbox{in}\ \Omega\times(0, T), 
 \\[2mm]
         \eta\partial_{\nu}\theta + \alpha_{\Gamma}(\theta-\theta_{\Gamma}) = 
         \partial_{\nu}\mu = \partial_{\nu}\varphi = 0                                   
         & \mbox{on}\ \Gamma\times(0, T),
 \\[2mm]
        (\ln\theta)(0) = \ln\theta_0,\ \varphi(0)=\varphi_0                                          
         & \mbox{in}\ \Omega, 
     \end{cases}
 \end{equation*}
%
%-----------------------------------------------------------------
%
%               Introduction of (Ptau)
%
%-----------------------------------------------------------------
 \begin{equation*}\tag*{(P)$_{\tau}$}\label{Ptau}
     \begin{cases}
         \partial_t (c_{s}\ln\theta_{\tau} + \lambda(\varphi_{\tau})) 
         - \eta\Delta\theta_{\tau} = f   
         & \mbox{in}\ \Omega\times(0, T), 
 \\[2mm]
        \partial_{t}\varphi_{\tau} - \Delta \mu_{\tau} = 0 
         & \mbox{in}\ \Omega\times(0, T), 
 \\[2mm]
         \mu_{\tau} = \tau\partial_{t}\varphi_{\tau} 
         - \gamma\Delta\varphi_{\tau} + \xi_{\tau} + \sigma'(\varphi_{\tau}) 
         - \lambda'(\varphi_{\tau})\theta_{\tau},\ \xi_{\tau} \in \beta(\varphi_{\tau})   
         & \mbox{in}\ \Omega\times(0, T), 
 \\[2mm]
         \eta\partial_{\nu}\theta_{\tau} 
         + \alpha_{\Gamma}(\theta_{\tau}-\theta_{\Gamma}) = 
         \partial_{\nu}\mu_{\tau} = \partial_{\nu}\varphi_{\tau} = 0                                   
         & \mbox{on}\ \Gamma\times(0, T),
 \\[2mm]
        (\ln\theta_{\tau})(0) = \ln\theta_0,\ \varphi_{\tau}(0)=\varphi_0                                          
         & \mbox{in}\ \Omega, 
     \end{cases}
 \end{equation*}
where $\Omega$ is a bounded domain in $\Rd$ ($d = 1, 2, 3$)
with smooth boundary $\Gamma:=\partial\Omega$. 
Moreover, we deal with the following conditions (C1)-(C7):  
%-----------------------------------------------------------------
%
%                      Assumption
%
%-----------------------------------------------------------------
%
\begin{enumerate} 
\setlength{\itemsep}{0mm}
\item[(C1)] $\beta \subset \mathbb{R}\times\mathbb{R}$                                
is a maximal monotone graph with effective domain $D(\beta)$ 
such that $\mbox{Int}\,D(\beta) \neq \emptyset$,   
and $\beta(r) = \partial\widehat{\beta}(r)$, where 
$\partial\widehat{\beta}$ denotes the subdifferential of 
a proper lower semicontinuous convex function 
$\widehat{\beta} : \mathbb{R} \to [0, +\infty]$ 
which has the effective domain $D(\widehat{\beta})$ and 
satisfies $\widehat{\beta}(0) = 0$.    
\item[(C2)] $\displaystyle\lim_{|r|\to+\infty}\frac{\widehat{\beta}(r)}{|r|^2}=+\infty$. 
\item[(C3)] $\sigma, \lambda \in C^1(\mathbb{R})$ and 
                $\sigma'$ and $\lambda'$ are Lipschitz continuous. 
\item[(C4)]  $\alpha_{\Gamma} \in L^{\infty}(\Gamma)$ and 
                 there exist positive constants $\alpha^{*}, \alpha_{*}$ such that 
                 $$
                 \alpha_{*} \leq \alpha_{\Gamma} \leq \alpha^{*}\quad 
                 \mbox{a.e.\ on}\ \Gamma. 
                 $$
\item[(C5)] $f \in L^2(0, T; L^2(\Omega))$.  
\item[(C6)] $\varphi_0 \in H^1(\Omega)$ 
                 and $\widehat{\beta}(\varphi_{0}) \in L^1(\Omega)$; 
                 moreover, the mean value 
                 $m_{0} := \frac{1}{|\Omega|}\int_{\Omega}\varphi_{0}$ 
                 lies in $\mbox{Int}\,D(\beta)$. 
\item[(C7)] $\theta_{\Gamma} \in L^{\infty}(\Gamma\times(0, T))$, 
                $\theta_{0} \in L^{\infty}(\Omega)$   
                and there exist positive constants $\theta^{*}, \theta_{*}$ such that 
                $$
                \theta_{*} \leq \theta_{\Gamma} \leq \theta^{*} 
                \quad\mbox{a.e.\ on}\ \Gamma\times(0, T) 
                \qquad \mbox{and}\qquad  
                \theta_{*} \leq \theta_0 \leq \theta^{*} 
                \quad\mbox{a.e.\ on}\ \Omega.  
                $$
\end{enumerate} 
Please note that a consequence of (C1) is that 
$0 \in \beta(0)$ since $0$ is a minimum for $\hat{\beta}$.

\smallskip

%-----------------------------------------------------------------
%
%                      Notations
%
%-----------------------------------------------------------------
%
\pier{Let us define the Hilbert spaces 
   $$
   H:=L^2(\Omega), \quad V:=H^1(\Omega)
   $$
 with inner products 
 \begin{align*} 
 &(u_{1}, u_{2})_{H}:=\int_{\Omega}u_{1}u_{2}\,dx \quad  (u_{1}, u_{2} \in H), \\
 &(v_{1}, v_{2})_{V}:=
 \int_{\Omega}\nabla v_{1}\cdot\nabla v_{2}\,dx + \int_{\Omega} v_{1}v_{2}\,dx \quad 
 (v_{1}, v_{2} \in V),
\end{align*}
 respectively,
 and with the related Hilbertian norms.}
 % $\|u\|_{H}:=(u, u)_{H}^{1/2}$ ($u\in H$) and 
 %$\|v\|_{V}:=(v, v)_{V}^{1/2}$ ($v\in V$), respectively.  
 Moreover, \pier{we use the notation}
   $$
   W:=\bigl\{z\in H^2(\Omega)\ |\ \partial_{\nu}z = 0 \quad 
   \mbox{a.e.\ on}\ \partial\Omega\bigr\}.
   $$ 
 The notation $V^{*}$ denotes the dual space of $V$ with 
 duality pairing $\langle\cdot, \cdot\rangle_{V^*, V}$. 
 Moreover, in this paper, the bijective mapping $F : V \to V^{*}$ and 
 the inner product in $V^{*}$ are defined as 
    \begin{align}
    &\langle Fv_{1}, v_{2} \rangle_{V^*, V} := 
    (v_{1}, v_{2})_{V} \quad (v_{1}, v_{2}\in V),   
    \label{defF}
    \\[1mm]
    &(v_{1}^{*}, v_{2}^{*})_{V^{*}} := 
    \left\langle v_{1}^{*}, F^{-1}v_{2}^{*} 
    \right\rangle_{V^*, V} 
    \quad (v_{1}^{*}, v_{2}^{*}\in V^{*}).   
    \label{innerVstar}
    \end{align}
This article employs the Hilbert space 
  $$
  V_{0}:=\left\{ z \in H^1(\Omega)\ \Big{|} \ \int_{\Omega} z = 0 \right\}   
  $$
 with inner product %(the same as in $V$)  
 \begin{align*} 
 (v_{1}, v_{2})_{V_{0}}:=
 \int_{\Omega}\nabla v_{1}\cdot\nabla v_{2}\,dx %+ \int_{\Omega} v_{1}v_{2}\,dx 
 \quad (v_{1}, v_{2} \in V_{0}) 
\end{align*}
 and with the related Hilbertian norm.
 The notation $V_{0}^{*}$ denotes the dual space of $V_{0}$ with 
 duality pairing $\langle\cdot, \cdot\rangle_{V_{0}^*, V_{0}}$. 
 Moreover, in this paper, the bijective mapping ${\cal N} : V_{0}^{*} \to V_{0}$ and 
 the inner product in $V_{0}^{*}$ are specified by  
    \begin{align}
    &\langle v^{*}, v \rangle_{V_{0}^*, V_{0}} := 
    \int_{\Omega} \nabla {\cal N}v^{*} \cdot \nabla v 
    \quad (v^{*} \in V_{0}^*, v \in V_{0}),   
    \label{defN}
    \\[1mm]
    &(v_{1}^{*}, v_{2}^{*})_{V_{0}^{*}} := 
    \left\langle v_{1}^{*}, {\cal N}v_{2}^{*} 
    \right\rangle_{V_{0}^*, V_{0}} 
    \quad (v_{1}^{*}, v_{2}^{*}\in V_{0}^{*}).  
    \label{innerVzerostar}
    \end{align}

We define weak solutions of \ref{P} and \ref{Ptau} as follows. 

%%%%%%%%%%%%%%%%%%%%%%DefP%%%%%%%%%%%%%%%%%%%%%%%%%%%
\begin{df}
A quadruple $(\theta, \mu, \varphi, \xi)$ with 
\begin{align*}
&\theta \in L^2(0, T; V), \\ 
&\mu \in L^2(0, T; V), \\ 
&\varphi \in H^1(0, T; V^{*}) \cap L^{\infty}(0, T; V) \cap L^2(0, T; W), \\
&\xi \in L^2(0, T; H), \\ 
&c_{s}\ln\theta + \lambda(\varphi) \in H^1(0, T; V^{*}), \\ 
&\ln\theta, \lambda(\varphi) \in L^{\infty}(0, T; H) 
\end{align*}
is called a {\it weak solution} of \ref{P} if $(\theta, \mu, \varphi, \xi)$ satisfies 
\begin{align}
&\bigl\langle (c_{s}\ln\theta + \lambda(\varphi))_{t}, v \bigr\rangle_{V^{*}, V} 
  + \eta\int_{\Omega}\nabla\theta\cdot\nabla v 
  + \int_{\Gamma} \alpha_{\Gamma}\theta v 
\label{dfPsol1} \\ 
  &= (f, v)_{H} + \int_{\Gamma} \alpha_{\Gamma}\theta_{\Gamma}v 
       \quad \mbox{a.e.\ on}\ (0, T) \quad \mbox{for all}\ v \in V, \notag  
\\[3mm] 
&\langle \varphi_{t}, v \rangle_{V^{*}, V} 
   + \int_{\Omega}\nabla\mu\cdot\nabla v = 0 
     \quad \mbox{a.e.\ on}\ (0, T) \quad \mbox{for all}\ v \in V, \label{dfPsol2} 
\\[1.5mm] 
&\mu = -\gamma\Delta\varphi+\xi+\sigma'(\varphi)-\lambda'(\varphi)\theta,\ 
                                                                      \xi \in \beta(\varphi)   
         \quad \mbox{a.e.\ on}\ \Omega\times(0, T), \label{dfPsol3} 
\\[3mm] 
&(c_{s}\ln\theta+\lambda(\varphi))(0) =c_{s}\ln\theta_0+\lambda(\varphi_0),\
 \varphi(0)=\varphi_0                                          
        \quad \mbox{a.e.\ on}\ \Omega. \label{dfPsol4}
\end{align}
\end{df}
%%%%%%%%%%%%%%%%%%%%%%DefP%%%%%%%%%%%%%%%%%%%%%%%%%%%
%
%
%
%%%%%%%%%%%%%%%%%%%%%%DefPtau%%%%%%%%%%%%%%%%%%%%%%%%%%%
\begin{df}
A quadruple $(\theta_{\tau}, \mu_{\tau}, \varphi_{\tau}, \xi_{\tau})$ with 
\begin{align*}
&\theta_{\tau} \in L^2(0, T; V), \\ 
&\mu_{\tau} \in L^2(0, T; V), \\ 
&\varphi_{\tau} \in H^1(0, T; H) \cap L^{\infty}(0, T; V) \cap L^2(0, T; W), \\
&\xi_{\tau} \in L^2(0, T; H), \\ 
&\ln\theta_{\tau}, \lambda(\varphi_{\tau}) 
                                                    \in H^1(0, T; V^{*}) \cap L^{\infty}(0, T; H) 
\end{align*}
is called a {\it weak solution} of \ref{Ptau} 
if $(\theta_{\tau}, \mu_{\tau}, \varphi_{\tau}, \xi_{\tau})$ satisfies 
\begin{align}
&c_{s}\bigl\langle (\ln\theta_{\tau})_{t}, v \bigr\rangle_{V^{*}, V} 
  + \bigl\langle (\lambda(\varphi_{\tau}))_{t}, 
                                                      v \bigr\rangle_{V^{*}, V} 
  + \eta\int_{\Omega}\nabla\theta_{\tau}\cdot\nabla v 
  + \int_{\Gamma} \alpha_{\Gamma}\theta_{\tau} v 
\label{dfPtausol1} \\ 
  &= (f, v)_{H} + \int_{\Gamma} \alpha_{\Gamma}\theta_{\Gamma}v 
       \quad \mbox{a.e.\ on}\ (0, T) \quad \mbox{for all}\ v \in V, \notag  
\\[3mm] 
&(\partial_{t}\varphi_{\tau}, v )_{H}  
   + \int_{\Omega}\nabla\mu_{\tau}\cdot\nabla v = 0 
     \quad \mbox{a.e.\ on}\ (0, T) \quad \mbox{for all}\ v \in V, \label{dfPtausol2} 
\\[1.5mm] 
&\mu_{\tau} = \tau\partial_{t}\varphi_{\tau} - \gamma\Delta\varphi_{\tau} 
                    + \xi_{\tau} 
                    + \sigma'(\varphi_{\tau}) - \lambda'(\varphi_{\tau})\theta_{\tau},\ 
                                                           \xi_{\tau} \in \beta(\varphi_{\tau})   
         \quad \mbox{a.e.\ on}\ \Omega\times(0, T), \label{dfPtausol3} 
\\[3mm] 
&(\ln\theta_{\tau})(0)=\ln\theta_0,\
 \varphi_{\tau}(0)=\varphi_0                                          
        \quad \mbox{a.e.\ on}\ \Omega. \label{dfPtausol4}
\end{align}
\end{df}
%%%%%%%%%%%%%%%%%%%%%%DefPtau%%%%%%%%%%%%%%%%%%%%%%%%%%%

Now the main results read as follows.
%%%%%%%%%%%%%%%%%%Theorem1%%%%%%%%%%%%%%%%%%%%%%%%
%%%%%%%%%%%%%%%%%%%%%%%%%%%%%%%%%%%%%%%%%%%%%%%%%%
\begin{thm}\label{maintheorem1}
Assume that {\rm (C1)-(C7)} hold 
and let  $\overline{\tau}$ denote a fixed bound 
for the viscosity coefficient $\tau$.     
Then there is   
a weak solution $(\theta_{\tau}, \mu_{\tau}, \varphi_{\tau}, \xi_{\tau})$ 
of {\rm \ref{Ptau}} for all $\tau \in (0, \overline{\tau})$.  
Moreover, 
there exists a constant $M>0$  
depending only on the data such that   
\begin{align*}
&\tau\|\partial_{t}\varphi_{\tau}\|_{L^2(0, T; H)}^2 
+ \|\varphi_{\tau}\|_{L^{\infty}(0, T; V)}^2 
+ \|\theta_{\tau}\|_{L^2(0, T; V)}^2 
\leq M, \\ 
&\|\partial_{t}\varphi_{\tau}\|_{L^2(0, T; V^{*})}^2 
  + \|\mu_{\tau}\|_{L^2(0, T; V)}^2 + \|\xi_{\tau}\|_{L^2(0, T; H)}^2 
  + \|\varphi_{\tau}\|_{L^2(0, T; W)}^2 
\leq M, \\ 
&\|c_{s}\ln\theta_{\tau}+\lambda(\varphi_{\tau})\|_{H^1(0, T; V^{*}) 
                                                                           \cap L^{\infty}(0, T; H)}^2 
\leq M, \\ 
&\|\ln\theta_{\tau}\|_{L^{\infty}(0, T; H)}^2 
   + \|\lambda(\varphi_{\tau})\|_{L^{\infty}(0, T; H)}^2 
  \leq M, \\ 
&\|(\ln\theta_{\tau})_{t}\|_{L^2(0, T; V^{*})}^2 
   + \|(\lambda(\varphi_{\tau}))_{t}\|_{L^2(0, T; V^{*})}^2 
   \leq M(1+\tau^{-1})   
\end{align*}
for all $\tau \in (0, \overline{\tau})$.   
\end{thm}
%%%%%%%%%%%%%%%%%%Theorem1%%%%%%%%%%%%%%%%%%%%%%%%
%%%%%%%%%%%%%%%%%%%%%%%%%%%%%%%%%%%%%%%%%%%%%%%%%%

%%%%%%%%%%%%%%%%%%Theorem2%%%%%%%%%%%%%%%%%%%%%%%%
%%%%%%%%%%%%%%%%%%%%%%%%%%%%%%%%%%%%%%%%%%%%%%%%%%
\begin{thm}\label{maintheorem2}
Assume that {\rm (C1)-(C7)} hold. 
Then there exists a weak solution of {\rm \ref{P}}. 
\end{thm}
%%%%%%%%%%%%%%%%%%Theorem2%%%%%%%%%%%%%%%%%%%%%%%%
%%%%%%%%%%%%%%%%%%%%%%%%%%%%%%%%%%%%%%%%%%%%%%%%%%

This paper is organized as follows. 
In Section \ref{Sec2} we consider a suitable approximation of \ref{Ptau}  
in terms of a parameter $\ep > 0$ and introduce a time discretization as well.  
Section \ref{Sec3} contains 
the proof of the existence for the discrete problem. 
In Section \ref{Sec4} we deduce uniform estimates for the time discrete solutions 
and consequently pass to the limit as the time step tends to zero. 
Additional a priori estimates, independent of the parameter $\ep$, 
are shown in Section \ref{Sec5} 
so that the existence of solutions for \ref{Ptau} is inferred via a limit procedure 
as $\ep \searrow 0$. 
Finally, in Section \ref{Sec6} we prove the existence of solutions to \ref{P} by 
taking the limit in \ref{Ptau} as $\tau \searrow 0$. 

%%==============================================================%%
%%==============                                  ==============%%
%%======                      Section2                    ======%%
%%====                                                      ====%%
%%==                                                          ==%%
%%====   Time discretization and main results  ====%%
%%======                                                  ======%%
%%==============                                  ==============%%
%%==============================================================%%

\section{Approximations}\label{Sec2}

To establish existence of solutions to \ref{Ptau} 
we consider the approximation 
%-----------------------------------------------------------------
%
%               Introduction of (Ptauep)
%
%-----------------------------------------------------------------
 \begin{equation*}\tag*{(P)$_{\ep}$}\label{Ptauep}
     \begin{cases}
         \partial_t (c_{s}\mbox{\rm Ln$_{\ep}$}(\theta_{\ep}) 
                                              + \lambda_{\ep}(\varphi_{\ep})) 
         - \eta\Delta\theta_{\ep} = f   
         & \mbox{in}\ \Omega\times(0, T), 
 \\[2mm]
        \partial_{t}\varphi_{\ep} - \Delta \mu_{\ep} = 0 
         & \mbox{in}\ \Omega\times(0, T), 
 \\[2mm]
         \mu_{\ep} = \tau\partial_{t}\varphi_{\ep} 
         - \gamma\Delta\varphi_{\ep} + \beta_{\ep}(\varphi_{\ep}) 
         + \sigma'(\varphi_{\ep}) 
         - \lambda_{\ep}'(\varphi_{\ep})\theta_{\ep}    
         & \mbox{in}\ \Omega\times(0, T), 
 \\[2mm]
         \eta\partial_{\nu}\theta_{\ep} 
         + \alpha_{\Gamma}(\theta_{\ep}-\theta_{\Gamma}) = 
         \partial_{\nu}\mu_{\ep} = \partial_{\nu}\varphi_{\ep} = 0                                   
         & \mbox{on}\ \Gamma\times(0, T),
 \\[2mm]
        (\mbox{\rm Ln$_{\ep}$}(\theta_{\ep}))(0) 
        = \mbox{\rm Ln$_{\ep}$}(\theta_0),\ 
        \varphi_{\ep}(0)=\varphi_0                                          
         & \mbox{in}\ \Omega, 
     \end{cases}
 \end{equation*}
where $\ep \in (0, 1]$, 
$\mbox{\rm Ln$_{\ep}$}(r):=\ep r + \ln_{\ep}(r)$, $r \in \mathbb{R}$, and 
$\ln_{\ep}$ is the Yosida approximation operator of $\ln$ on $\mathbb{R}$,   
$\beta_{\ep} : \mathbb{R} \to \mathbb{R}$ is the 
Yosida approximation operator of $\beta$ on $\mathbb{R}$,
and $\lambda_{\ep} : \mathbb{R} \to \mathbb{R}$ satisfies 
$\lambda_{\ep} \in C^1(\mathbb{R})$ and 
\begin{align}
&\lambda_{\ep}\ \mbox{and}\ \lambda_{\ep}'\ 
\mbox{are Lipschitz continuous}, \label{lamep1} \\[1mm] 
&|\lambda_{\ep}(0)| + |\lambda_{\ep}'(0)| 
  + \|\lambda''_{\ep}\|_{L^{\infty}(\mathbb{R})} \leq M_{\lambda} 
  \quad \mbox{for all}\ \ep \in (0, 1], \label{lamep2} \\[1mm] 
&\lambda_{\ep}(r) \to \lambda(r) \quad\mbox{and}\quad 
  \lambda_{\ep}'(r) \to \lambda'(r) \quad\mbox{as}\ \ep\searrow0\ 
  \mbox{for all}\ r \in \mathbb{R}, \label{lamep3}
\end{align}  
with some constant $M_{\lambda}>0$. 
\begin{remark}\label{remark.lambep}
A possible choice for $\lambda_{\ep}$ is  
$$
 \lambda_{\ep}(r)=
   \begin{cases}
   \lambda(\frac{1}{\ep}) + \lambda'(\frac{1}{\ep})\left(r - \frac{1}{\ep} \right) 
   & \mbox{if}\ 
   r > \frac{1}{\ep}, 
   \\[3mm]
   \lambda(r) 
   & \mbox{if}\ 
   |r| \leq \frac{1}{\ep}, 
   \\[3mm] 
    \lambda(-\frac{1}{\ep}) + \lambda'(-\frac{1}{\ep})\left(r + \frac{1}{\ep} \right) 
   & \mbox{if}\ 
   r < - \frac{1}{\ep}.  
   \end{cases}
$$
Indeed, we have that 
$$
 \lambda_{\ep}'(r)=
   \begin{cases}
   \lambda'(\frac{1}{\ep})  
   & \mbox{if}\ 
   r > \frac{1}{\ep}, 
   \\[3mm]
   \lambda'(r) 
   & \mbox{if}\ 
   |r| \leq \frac{1}{\ep}, 
   \\[3mm] 
    \lambda'(-\frac{1}{\ep})  
   & \mbox{if}\ 
   r < - \frac{1}{\ep}  
   \end{cases}
$$
and 
$$
 \lambda_{\ep}''(r)=
   \begin{cases}
   0  
   & \mbox{if}\ 
   r > \frac{1}{\ep}, 
   \\[3mm]
   \lambda''(r) 
   & \mbox{if}\ 
   |r| < \frac{1}{\ep}, 
   \\[3mm] 
   0   
   & \mbox{if}\ 
   r < - \frac{1}{\ep},   
   \end{cases}
$$
and hence we can confirm that 
$\lambda_{\ep} : \mathbb{R} \to \mathbb{R}$ is Lipschitz continuous, 
$\lambda_{\ep}' : \mathbb{R} \to \mathbb{R}$ is Lipschitz continuous and bounded. 
The properties \eqref{lamep1}-\eqref{lamep3} are satisfied because 
$\lambda_{\ep}(0)=\lambda(0)$, $\lambda_{\ep}'(0)=\lambda'(0)$  
and  
$\|\lambda_{\ep}''\|_{L^{\infty}(\mathbb{R})}
\leq\|\lambda''\|_{L^{\infty}(\mathbb{R})}$ for all $\ep \in (0, 1]$.    
\end{remark}
\begin{remark}\label{about.Lnep}
We have that the function $\mbox{Ln}_{\ep}$ is monotone and Lipschitz continuous 
(see, e.g., \cite[p.\ 28]{Brezis}) 
and satisfies the inequality 
$\mbox{Ln}_{\ep}'(r) \geq \ep$ for all $r \in \mathbb{R}$.  
\end{remark}
\begin{remark}\label{def.of.rhoep}
Let $\rho_{\ep} : \mathbb{R} \to \mathbb{R}$ be 
the resolvent operator of $\ln$ on $\mathbb{R}$.
Then $\rho_{\ep}$ is positive and  
$\rho_{\ep}(r)$ is the unique solution of 
the equation 
$\rho_{\ep}(r) + \ep\ln \rho_{\ep}(r) = r$ 
for any $r \in \mathbb{R}$. 
Thus we can emphasize that 
$\ln_{\ep}(r) = \ln \rho_{\ep}(r) = \frac{1}{\ep}(r-\rho_{\ep}(r))$ 
for all $r \in \mathbb{R}$.
\end{remark}
\begin{remark}\label{aboutbetahatep}
The function $\hat{\beta}_{\ep} : \mathbb{R} \to \mathbb{R}$ defined by 
$$
\hat{\beta}_{\ep}(r):=
\displaystyle\inf_{s\in\mathbb{R}}\left\{\frac{1}{2\ep}|r-s|^2+\hat{\beta}(s) \right\} 
\quad \mbox{for}\ r\in \mathbb{R} 
$$ 
is called the Moreau--Yosida regularization  
of $\hat{\beta}$, 
which has the identity 
$$
\hat{\beta}_{\ep}(r)=\frac{1}{2\ep}|r-J_{\ep}^{\beta}(r)|^2 
+ \hat{\beta}(J_{\ep}^{\beta}(r))
$$ 
for all $r\in\mathbb{R}$ and all $\ep >0$, 
where $J_{\ep}^{\beta}$ is the resolvent operator of $\beta$  
on $\mathbb{R}$.  
Moreover, we can infer that  
$$
\beta_{\ep}(r) = \partial\hat{\beta}_{\ep}(r) = \frac{d}{dr}\hat{\beta}_{\ep}(r), \quad 
0\leq \hat{\beta}_{\ep}(r) \leq \hat{\beta}(r)
$$
for all $r\in\mathbb{R}$ and all $\ep > 0$ 
(see, e.g., \cite[Theorem 2.9, p.\ 48]{Barbu2}). 
\end{remark}
\begin{remark}\label{Bhatepzero}
We can observe from Remark \ref{aboutbetahatep} that $\beta_{\ep}(0)=0$.  
Indeed, the inequalities $0\leq \hat{\beta}_{\ep}(0) \leq \hat{\beta}(0)$  
and the condition (C1) yield that $\hat{\beta}_{\ep}(0)=0$,   
whence we can derive that 
$$
0=\hat{\beta}_{\ep}(0)=\frac{1}{2\ep}|J_{\ep}^{\beta}(0)|^2 
+ \hat{\beta}(J_{\ep}^{\beta}(0)) 
\geq \frac{1}{2\ep}|J_{\ep}^{\beta}(0)|^2.   
$$
Thus we can verify that $J_{\ep}^{\beta}(0)=0$, 
which implies that    
$\beta_{\ep}(0)=0$ 
by the identity $\beta_{\ep}(r)=\frac{1}{\ep}(r-J_{\ep}^{\beta}(r))$. 
\end{remark}

The definition of weak solutions to \ref{Ptauep} is as follows.   

%
%%%%%%%%%%%%%%%%%%%%%%DefPtau%%%%%%%%%%%%%%%%%%%%%%%%%%%
\begin{df}
A triplet $(\theta_{\ep}, \mu_{\ep}, \varphi_{\ep})$ with 
\begin{align*}
&\theta_{\ep} \in L^2(0, T; V), \\ 
&\mu_{\ep} \in L^2(0, T; V), \\ 
&\varphi_{\ep} \in H^1(0, T; H) \cap L^{\infty}(0, T; V) \cap L^2(0, T; W), \\ 
&\mbox{\rm Ln$_{\ep}$}(\theta_{\ep}) 
\in H^1(0, T; V^{*}) \cap L^{\infty}(0, T; H), \\ 
&\lambda_{\ep}(\varphi_{\ep}) \in H^1(0, T; H) \cap L^{\infty}(0, T; V) 
\end{align*}
is called a {\it weak solution} of \ref{Ptauep} 
if $(\theta_{\ep}, \mu_{\ep}, \varphi_{\ep})$ satisfies 
\begin{align}
&c_{s}\bigl\langle (\mbox{\rm Ln$_{\ep}$}(\theta_{\ep}))_{t}, 
                                                                        v \bigr\rangle_{V^{*}, V} 
  + \bigl(\partial_{t}\lambda_{\ep}(\varphi_{\ep}), 
                                                      v \bigr)_{H} 
  + \eta\int_{\Omega}\nabla\theta_{\ep}\cdot\nabla v 
  + \int_{\Gamma} \alpha_{\Gamma}\theta_{\ep} v 
\label{dfPtauepsol1} \\ 
  &= (f, v)_{H} + \int_{\Gamma} \alpha_{\Gamma}\theta_{\Gamma} v 
       \quad \mbox{a.e.\ on}\ (0, T) \quad \mbox{for all}\ v \in V, \notag  
\\[3mm] 
&(\partial_{t}\varphi_{\ep}, v)_{H}  
   + \int_{\Omega}\nabla\mu_{\ep}\cdot\nabla v = 0 
     \quad \mbox{a.e.\ on}\ (0, T) \quad \mbox{for all}\ v \in V, \label{dfPtauepsol2} 
\\[1.5mm] 
&\mu_{\ep} = \tau\partial_{t}\varphi_{\ep} 
                    - \gamma\Delta\varphi_{\ep} 
                    + \beta_{\ep}(\varphi_{\ep}) 
                    + \sigma'(\varphi_{\ep}) 
                    - \lambda'(\varphi_{\ep})\theta_{\ep}   
         \quad \mbox{a.e.\ on}\ \Omega\times(0, T), \label{dfPtauepsol3} 
\\[3mm] 
&(\mbox{\rm Ln$_{\ep}$}(\theta_{\ep}))(0)=\mbox{\rm Ln$_{\ep}$}(\theta_0),\
 \varphi_{\ep}(0)=\varphi_0                                          
        \quad \mbox{a.e.\ on}\ \Omega. \label{dfPtauepsol4}
\end{align}
\end{df}
%%%%%%%%%%%%%%%%%%%%%%DefPtau%%%%%%%%%%%%%%%%%%%%%%%%%%%

\begin{lem}\label{existPtauep}
Assume that {\rm (C1)-(C7)} hold. 
Then there exists $\ep_{0} \in (0, 1]$ such that 
there is  
a weak solution 
%$(\theta_{\ep}, \mu_{\ep}, \varphi_{\ep})$ 
of {\rm \ref{Ptauep}} for all $\ep \in (0, \ep_{0})$. 
\end{lem}

%\begin{remark}\label{about.lambdaep}
%The inequality 
%$$
%|\ln_{\ep} r| \leq |\ln r|
%$$ 
%holds for all $r>0$ and all $\ep>0$  
%(see, e.g., \blue{\cite[p.\ 28]{?}}).   
%\end{remark}

To prove Lemma \ref{existPtauep} we employ a time discretization scheme. 
More precisely, we will deal with the following problem: 
find $(\theta_{n+1}, \mu_{n+1}, \varphi_{n+1}) \in V \times W \times W$ such that  
\begin{equation*}\tag*{(P)$_{n}$}\label{Ptauepn}
     \begin{cases}
         c_{s}\delta_{h}u_{n} + \lambda_{\ep}'(\varphi_{n})\delta_{h}\varphi_{n}
         - \eta\Delta\theta_{n+1} = f_{n+1}   
         & \mbox{in}\ \Omega, 
 \\[2mm]
        \delta_{h}\varphi_{n} 
         + h \delta_{h}\mu_{n} - \Delta \mu_{n+1} 
         = 0  
         & \mbox{in}\ \Omega, 
 \\[2mm]
         \mu_{n+1} = 
         \tau\delta_{h}\varphi_{n} 
         - \gamma\Delta\varphi_{n+1} 
         + \beta_{\ep}(\varphi_{n+1}) 
         + \sigma'(\varphi_{n+1}) 
         - \lambda_{\ep}'(\varphi_{n})\theta_{n+1}    
         & \mbox{in}\ \Omega, 
 \\[2mm]
         \eta\partial_{\nu}\theta_{n+1} 
         + \alpha_{\Gamma}(\theta_{n+1}-\theta_{\Gamma, n+1}) = 
         \partial_{\nu}\mu_{n+1} = \partial_{\nu}\varphi_{n+1} = 0                                   
         & \mbox{on}\ \Gamma
     \end{cases}
\end{equation*}
for $n=0, ... , N-1$, where $h=\frac{T}{N}$, $N \in \mathbb{N}$, 
\begin{align}\label{uj}
&u_{j} := \mbox{\rm Ln$_{\ep}$}(\theta_{j}) 
\end{align}
for $j=0, 1, ..., N$,   
\begin{align}\label{delta1}
&\delta_{h}u_{n} := \frac{u_{n+1}-u_{n}}{h},\ 
\delta_{h}\varphi_{n} 
:= \frac{\varphi_{n+1} - \varphi_{n}}{h},\ 
\delta_{h}\mu_{n} 
:= \frac{\mu_{n+1} - \mu_{n}}{h},  
\end{align}
$f_{k} := \frac{1}{h}\int_{(k-1)h}^{kh} f(s)\,ds$, 
and $\theta_{\Gamma, k} := \frac{1}{h}\int_{(k-1)h}^{kh} \theta_{\Gamma}(s)\,ds$ 
for $k = 1, ... , N$. 
Note that, in order to solve the above system, 
we also need an initial value $\mu_{0}$, which is no present in (C1)-(C7), 
and it is up to our choice. For simplicity, we take 
\begin{align}\label{muzerozero}
\mu_{0} = 0. 
\end{align}
Also, putting 
\begin{align}
& \hat{u}_{h}(0) := \ln_{\ep}(\theta_0),\ 
   \partial_t \hat{u}_{h} (t) := \delta_{h}u_{n}, \label{hat1} 
\\[2mm]  
&\hat{\varphi}_{h}(0) := \varphi_0,\ 
   \partial_t \hat{\varphi}_{h} (t) := \delta_{h}\varphi_{n}, \label{hat2}
\\[2mm]
&\hat{\mu}_{h}(0) := \mu_0,\ 
   \partial_t \hat{\mu}_{h} (t) 
                                             := \delta_{h}\mu_{n}, \label{hat3} 
\\[2mm] 
&\overline{u}_{h}(t) := u_{n+1},\ 
  \overline{\theta}_{h} (t) := \theta_{n+1},\ 
  \underline{\theta}_{h} (t) := \theta_{n},\ 
  \overline{f}_h (t) := f_{n+1},\ \overline{\theta_{\Gamma}}_h (t) 
                                                               := \theta_{\Gamma, n+1},   
\label{line1} 
\\[2mm] 
&\overline{\varphi}_{h} (t) := \varphi_{n+1},\ 
\underline{\varphi}_{h} (t) := \varphi_{n},\ 
\overline{\mu}_{h} (t) := \mu_{n+1}  
\label{line2}   
\end{align}
for a.a.\ $t \in (nh, (n+1)h)$, $n=0, ..., N-1$, 
we can rewrite \ref{Ptauepn} as  
\begin{equation*}\tag*{(P)$_{h}$}\label{Ptaueph}
     \begin{cases}
         c_{s}\partial_t \hat{u}_{h}  
         + \lambda_{\ep}'(\underline{\varphi}_{h})\partial_t \hat{\varphi}_{h} 
         - \eta\Delta\overline{\theta}_{h} = \overline{f}_h    
         & \mbox{in}\ \Omega\times(0, T), 
 \\[2mm]
        \partial_{t}\hat{\varphi}_{h} 
         + h\partial_{t}\hat{\mu}_{h} 
         - \Delta \overline{\mu}_{h} = 0 
         & \mbox{in}\ \Omega\times(0, T), 
 \\[2mm]
         \overline{\mu}_{h} = \tau \partial_{t}\hat{\varphi}_{h} 
         - \gamma\Delta\overline{\varphi}_{h} 
         + \beta_{\ep}(\overline{\varphi}_{h}) 
         + \sigma'(\overline{\varphi}_{h}) 
         - \lambda_{\ep}'(\underline{\varphi}_{h})\overline{\theta}_{h}
         & \mbox{in}\ \Omega\times(0, T), 
 \\[2mm]
         \overline{u}_{h} = \mbox{\rm Ln$_{\ep}$}(\overline{\theta}_{h})   
         & \mbox{in}\ \Omega\times(0, T),  
 \\[2mm]
         \eta\partial_{\nu}\overline{\theta}_{h} 
         + \alpha_{\Gamma}(\overline{\theta}_{h}
                                                   -\overline{\theta_{\Gamma}}_h) = 
         \partial_{\nu}\overline{\mu}_{h} 
         = \partial_{\nu}\overline{\varphi}_{h} = 0                                   
         & \mbox{on}\ \Gamma\times(0, T),
 \\[2mm]
        \hat{u}_{h}(0) = \mbox{\rm Ln$_{\ep}$}(\theta_0),\ 
        \hat{\varphi}_{h}(0) = \varphi_0,\    
        \hat{\mu}_{h}(0) = \mu_0 = 0                                    
         & \mbox{in}\ \Omega.  
     \end{cases}
 \end{equation*}

\begin{remark}
On account of \eqref{muzerozero} and \eqref{hat1}-\eqref{line2}, 
the reader can check directly the following properties: 
\begin{align}
&\|\widehat{\varphi}_h\|_{L^2(0, T; H)}^2 
\leq h\|\varphi_0\|_{H}^2 + 2\|\overline{\varphi}_h\|_{L^2(0, T; H)}^2, \label{rem1} 
\\[2mm] 
&\|\widehat{\varphi}_h\|_{L^{\infty}(0, T; V)} 
= \max\{\|\varphi_{0}\|_{V}, \|\overline{\varphi}_h\|_{L^{\infty}(0, T; V)}\}, 
\label{rem2} 
\\[2mm]
&\|\widehat{\mu}_h\|_{L^2(0, T; H)}^2 
\leq 2\|\overline{\mu}_h\|_{L^2(0, T; H)}^2, \label{rem3} 
\\[2mm] 
&\|\widehat{\mu}_h\|_{L^{\infty}(0, T; H)} 
= \|\overline{\mu}_h\|_{L^{\infty}(0, T; H)}, \label{rem4}
\\[2mm] 
&\|\widehat{u}_h\|_{L^2(0, T; V^{*})}^2 
\leq h\|u_{0}\|_{V^{*}}^2 + 2\|\overline{u}_h\|_{L^2(0, T; V^{*})}^2, \label{rem5} 
\\[2mm] 
&\|\widehat{u}_h\|_{L^{\infty}(0, T; H)} 
= \max\{\|u_{0}\|_{H}, \|\overline{u}_h\|_{L^{\infty}(0, T; H)}\}, \label{rem6} 
\\[2mm] 
&\|\overline{\varphi}_h - \widehat{\varphi}_h\|_{L^2(0, T; H)}^2 
= \frac{h^2}{3}\|\partial_t \widehat{\varphi}_h\|_{L^2(0, T; H)}^2, \label{rem7}
\\[2mm]
&\|\overline{u}_h - \widehat{u}_h\|_{L^2(0, T; V^{*})}^2 
= \frac{h^2}{3}\|\partial_t \widehat{u}_h\|_{L^2(0, T; V^{*})}^2, \label{rem8}
\\[2mm]
&\underline{\varphi}_{h} = \overline{\varphi}_{h} - h \partial_t \widehat{\varphi}_h.  
\label{rem9} 
%\\
%&\underline{\mu}_{h} = \overline{\mu}_{h} - h \partial_t \widehat{\mu}_h, 
%\label{rem10} 
\end{align}
\end{remark}

\begin{df}
For $n=0, ... , N-1$, a triplet 
$(\theta_{n+1}, \mu_{n+1}, \varphi_{n+1})$ with 
\begin{align*}
\theta_{n+1}\in V,\ \mu_{n+1}, \varphi_{n+1} \in W  
\end{align*}
is called a {\it weak solution} of \ref{Ptauepn} 
if $(\theta_{n+1}, \mu_{n+1}, \varphi_{n+1})$ satisfies 
\begin{align}
&c_{s}(\delta_{h}u_{n}, v)_{H} 
  + (\lambda_{\ep}'(\varphi_{n})\delta_{h}\varphi_{n}, v)_{H} 
  + \eta\int_{\Omega}\nabla\theta_{n+1}\cdot\nabla v 
  + \int_{\Gamma} 
                 \alpha_{\Gamma}\theta_{n+1} v 
\label{dfPtauepnsol1} \\ 
  &= (f_{n+1}, v)_{H} + \int_{\Gamma} 
                 \alpha_{\Gamma} \theta_{\Gamma, n+1} v 
       \quad \mbox{for all}\ v \in V, \notag  
\\[3mm] 
&\delta_{h}\varphi_{n} + h\delta_{h}\mu_{n}
   - \Delta\mu_{n+1} = 0 
     \quad \mbox{a.e.\ on}\ \Omega, \label{dfPtauepnsol2} 
\\[1.5mm]  
&\mu_{n+1} = 
         \tau\delta_{h}\varphi_{n} 
         - \gamma\Delta\varphi_{n+1} \label{dfPtauepnsol3} 
         + \beta_{\ep}(\varphi_{n+1}) 
         + \sigma'(\varphi_{n+1}) 
         - \lambda_{\ep}'(\varphi_{n})\theta_{n+1}    
         \quad \mbox{a.e.\ on}\ \Omega, 
\end{align}
where 
\begin{align}\label{uj.insystem}
u_{j} = \mbox{\rm Ln$_{\ep}$}(\theta_{j}),  
%w_{j} = \lambda_{\ep}(\varphi_{j}) 
\quad j=0, 1, ... , N.   
\end{align}
\end{df}

\begin{lem}\label{existPtauepn}
Assume that {\rm (C1)-(C7)} hold. Then 
for all $h$ such that 
$$
0< h < \min\left\{\frac{\tau}{2\|\sigma''\|_{L^{\infty}(\mathbb{R})}}, 
              \frac{c_{s}\ep\tau}{2\|\lambda_{\ep}'\|_{L^{\infty}(\mathbb{R})}^2} \right\} 
$$
there exists a unique weak solution of {\rm \ref{Ptauepn}} for $n=0, ..., N-1$. 
\end{lem}

\vspace{10pt}

%%==============================================================%%
%%==============                                  ==============%%
%%======                      Section3                    ======%%
%%====                                                      ====%%
%%==                                                          ==%%
%%====   Existence of discrete solution           ====%%
%%======                                                  ======%%
%%==============                                  ==============%%
%%==============================================================%%

\section{Existence of time discrete solutions}\label{Sec3}

In this section we will prove Lemma \ref{existPtauepn}.  
\begin{lem}\label{ellipticeq1}
For all $g^{*} \in V^{*}$ and all $h > 0$ there exists 
a unique function $\theta \in V$ satisfying the identity 
$$
c_{s} \int_{\Omega} \mbox{\rm Ln$_{\ep}$}(\theta)\,v  
   + \eta h\int_{\Omega} \nabla\theta\cdot\nabla v 
   + h\int_{\Gamma} \alpha_{\Gamma} \theta v  
= \langle g^{*}, v \rangle_{V^{*}, V} 
$$
for all $v \in V$. 
\end{lem}
\begin{proof}
We define $\Phi : V \to V^{*}$ by 
$$
\langle \Phi\theta, v \rangle_{V^{*}, V} 
:= c_{s} \int_{\Omega} \mbox{\rm Ln$_{\ep}$}(\theta)\,v  
   + \eta h\int_{\Omega} \nabla\theta\cdot\nabla v 
   + h\int_{\Gamma} \alpha_{\Gamma} \theta v  
\quad \mbox{for}\ \theta, v \in V.   
$$
Then this operator $\Phi : V \to V^{*}$ is monotone, continuous and coercive 
for all $h > 0$. 
Indeed, 
note that there exist constants $C^{*}, C_{*} > 0$ such that 
\begin{align}\label{keyineq}
C_{*}\bigl(\|\nabla z\|_{H}^2 + \|z\|_{L^2(\Gamma)}^2 \bigr) 
\leq \|z\|_{V}^2 
\leq C^{*}\bigl(\|\nabla z\|_{H}^2 + \|z\|_{L^2(\Gamma)}^2 \bigr) 
\end{align}
for all $z \in V$ (see, e.g., \cite[p.\ 20]{N1967}).  
The first inequality in \eqref{keyineq} can be obtained by the trace theorem.   
Hence we see from (C4), \eqref{keyineq} and Remark \ref{about.Lnep}  
that 
\begin{align*}
\langle\Phi\theta-\Phi\overline{\theta}, \theta-\overline{\theta}\rangle_{V^{*}, V} 
&= c_{s} \int_{\Omega} 
          (\mbox{\rm Ln$_{\ep}$}(\theta)-\mbox{\rm Ln$_{\ep}$}\overline{\theta})
                                                                             (\theta-\overline{\theta})
   + \eta h\int_{\Omega} |\nabla(\theta-\overline{\theta})|^2 
   + h\int_{\Gamma} \alpha_{\Gamma} (\theta-\overline{\theta})^2 
\\
&\geq \min\{\eta, \alpha_{*}\}h
               \Bigl(\int_{\Omega} |\nabla (\theta-\overline{\theta})|^2 
                                         + \int_{\Gamma} |\theta-\overline{\theta}|^2 \Bigr) 
\\
&\geq \min\{\eta, \alpha_{*}\}\frac{h}{C^{*}}\|\theta-\overline{\theta}\|_{V}^2,     
\end{align*} 
\begin{align*}
&|\langle \Phi\theta - \Phi\overline{\theta}, v \rangle_{V^{*}, V}| 
\\ 
&= \left| 
   c_{s} \int_{\Omega} 
       (\mbox{\rm Ln$_{\ep}$}(\theta)-\mbox{\rm Ln$_{\ep}$}\overline{\theta})v
   + \eta h\int_{\Omega} \nabla(\theta-\overline{\theta}) \cdot \nabla v   
   + h\int_{\Gamma} \alpha_{\Gamma} (\theta-\overline{\theta})v 
     \right|
\\ 
&\leq c_{s}\|\mbox{\rm Ln$_{\ep}$}'\|_{L^{\infty}(\mathbb{R})}
                         \|\theta-\overline{\theta}\|_{H}\|v\|_{H} 
         + \eta h \|\nabla(\theta-\overline{\theta})\|_{H}\|\nabla v\|_{H} 
         + \alpha^{*}h 
                \|\theta-\overline{\theta}\|_{L^2(\Gamma)}\|v\|_{L^2(\Gamma)} 
\\ 
&\leq \left(c_{s}\|\mbox{\rm Ln$_{\ep}$}'\|_{L^{\infty}(\mathbb{R})} 
                + \frac{(\eta + \alpha^{*})h}{C_{*}} \right)
                                     \|\theta-\overline{\theta}\|_{V}\|v\|_{V} 
\end{align*}
and 
\begin{align*}
\langle \Phi\theta - \mbox{\rm Ln$_{\ep}$}(0), \theta \rangle_{V^{*}, V}  
&=  c_{s} \int_{\Omega} 
             (\mbox{\rm Ln$_{\ep}$}(\theta)-\mbox{\rm Ln$_{\ep}$}(0))(\theta-0)  
   + \eta h\int_{\Omega} |\nabla\theta|^2  
   + h\int_{\Gamma} \alpha_{\Gamma} |\theta|^2 
\\ 
&\geq \min\{\eta, \alpha_{*}\}h
               \Bigl(\int_{\Omega} |\nabla \theta|^2 + \int_{\Gamma} |\theta|^2 \Bigr) 
\\ 
&\geq \min\{\eta, \alpha_{*}\}\frac{h}{C^{*}}\|\theta\|_{V}^2            
\end{align*}
for all $\theta, \overline{\theta}, v \in V$. 
Thus the operator $\Phi : V \to V^{*}$ is surjective for all $h > 0$ 
(see, e.g., \cite[p.\ 37]{Barbu2}),  
which leads to Lemma \ref{ellipticeq1}.  
\end{proof}

\begin{lem}\label{ellipticeq2}
For all $g \in H$ 
and all $h \in \left(0, \frac{\tau}{\|\sigma''\|_{L^{\infty}(\mathbb{R})}} \right)$ 
there exists 
a unique solution $\varphi \in W$ of the equation 
$\tau\varphi + (1-\Delta)^{-1}\varphi - \gamma h \Delta\varphi 
+ h\beta_{\ep}(\varphi) + h\sigma'(\varphi) = g$, 
where $(1-\Delta)^{-1}$ is the inverse operator of 
$v \in W \mapsto v-\Delta v \in H$.  
\end{lem}
\begin{proof}
We define $\Psi : V \to V^{*}$ by 
$$
\langle \Psi\varphi, v \rangle_{V^{*}, V} 
:= \tau(\varphi, v)_{H} 
   + ((1-\Delta)^{-1}\varphi, v)_{H} 
   + \gamma h \int_{\Omega} \nabla \varphi \cdot \nabla v 
   + h(\beta_{\ep}(\varphi), v)_{H} 
   + h(\sigma'(\varphi), v)_{H} 
$$
for $\varphi, v \in V$. 
Then this operator $\Psi : V \to V^{*}$ is monotone, continuous and coercive 
for all $h \in \left(0, \frac{\tau}{\|\sigma''\|_{L^{\infty}(\mathbb{R})}} \right)$. 
Indeed, 
Lipschitz continuity of $\beta_{\ep}$ and $(1-\Delta)^{-1}$ with 
Lipschitz constants $1/\ep$ and $1$, respectively, 
the monotonicity of $\beta_{\ep}$ and $(1-\Delta)^{-1}$, 
and Remark \ref{Bhatepzero}  
yield that 
\begin{align*}
&\langle\Psi\varphi-\Psi\overline{\varphi}, 
                                         \varphi-\overline{\varphi}\rangle_{V^{*}, V} 
\geq \min\{\tau - \|\sigma''\|_{L^{\infty}(\mathbb{R})}h, \gamma h\}
                                                           \|\varphi-\overline{\varphi}\|_{V}^2,      
\\[3.5mm] 
&|\langle \Psi\varphi-\Psi\overline{\varphi}, v \rangle_{V^{*}, V}| 
\leq \left(\tau + 1 + \gamma h + \frac{h}{\ep} 
                           + \|\sigma''\|_{L^{\infty}(\mathbb{R})}h \right)
                                     \|\varphi-\overline{\varphi}\|_{V}\|v\|_{V}, 
\\[3.5mm] 
&\langle\Psi\varphi-h\sigma'(0), \varphi\rangle_{V^{*}, V} 
\geq \min\{\tau - \|\sigma''\|_{L^{\infty}(\mathbb{R})}h, \gamma h\}
                                                                                   \|\varphi\|_{V}^2      
\end{align*}
for all $\varphi, \overline{\varphi}, v \in V$. 
Therefore the operator $\Psi : V \to V^{*}$ is surjective 
for all $h \in \left(0, \frac{\tau}{\|\sigma''\|_{L^{\infty}(\mathbb{R})}} \right)$ 
(see, e.g., \cite[p.\ 37]{Barbu2}), 
and hence 
we can conclude from 
the elliptic regularity theory that 
Lemma \ref{ellipticeq2} holds. 
\end{proof}

\begin{prlem3.2}
The system \eqref{dfPtauepnsol1}-\eqref{uj.insystem} 
can be written as 
\begin{align}
&c_{s}(\mbox{\rm Ln$_{\ep}$}(\theta_{n+1}), v)_{H} 
  + \eta h\int_{\Omega}\nabla\theta_{n+1}\cdot\nabla v 
  + h\int_{\Gamma} 
           \alpha_{\Gamma}\theta_{n+1}v 
\label{writtendfPtauepnsol1} \\ 
  &= h(f_{n+1}, v)_{H} 
       + c_{s}(\mbox{\rm Ln$_{\ep}$}(\theta_{n}), v)_{H} 
       + h\int_{\Gamma}\alpha_{\Gamma}\theta_{\Gamma, n+1}v  
\notag  \\ 
  &\hspace{6cm} 
   + \bigl(\lambda_{\ep}'(\varphi_{n})(\varphi_{n}-\varphi_{n+1}), v\bigr)_{H} 
       \quad \mbox{for all}\ v \in V, \notag  
\\[5mm] 
&\mu_{n+1} 
= (1-\Delta)^{-1}\mu_{n} 
  + \frac{1}{h}(1-\Delta)^{-1}\varphi_{n} 
  - \frac{1}{h}(1-\Delta)^{-1}\varphi_{n+1} 
     \quad \mbox{a.e.\ on}\ \Omega, \label{writtendfPtauepnsol2} 
\\[3.5mm]  
&\tau\varphi_{n+1} + (1-\Delta)^{-1}\varphi_{n+1} 
   - \gamma h\Delta\varphi_{n+1} 
   + h\beta_{\ep}(\varphi_{n+1}) 
   + h\sigma'(\varphi_{n+1})
\label{writtendfPtauepnsol3} 
\\ \notag
&= \tau\varphi_{n} + h(1-\Delta)^{-1}\mu_{n} 
     + (1-\Delta)^{-1}\varphi_{n} 
     + h\lambda_{\ep}'(\varphi_{n})\theta_{n+1}    
         \quad \mbox{a.e.\ on}\ \Omega   
\end{align}
for $n=0, ... , N-1$. 
To prove Lemma \ref{existPtauepn} 
it suffices to establish existence and uniqueness of solutions to 
the system \eqref{writtendfPtauepnsol1}-\eqref{writtendfPtauepnsol3} 
in the case that $n=0$, for a general $\mu_{0} \in H$. 
Let $h \in \left(0, \frac{\tau}{2\|\sigma''\|_{L^{\infty}(\mathbb{R})}} \right)$. 
Then Lemma \ref{ellipticeq1} implies that  
for all $\varphi \in H$ there exists a unique function $\overline{\theta} \in V$ 
such that 
\begin{align}\label{formapA}
&c_{s}(\mbox{\rm Ln$_{\ep}$}(\overline{\theta}), v)_{H} 
  + \eta h\int_{\Omega}\nabla\overline{\theta}\cdot\nabla v 
  + h\int_{\Gamma} 
           \alpha_{\Gamma}\overline{\theta}v 
\\ \notag 
  &= h(f_{1}, v)_{H} 
       + c_{s}(\mbox{\rm Ln$_{\ep}$}(\theta_{0}), v)_{H} 
       + h\int_{\Gamma}\alpha_{\Gamma}\theta_{\Gamma, 1}v  
\\ \notag  
  &\hspace{6cm} 
   + \bigl(\lambda_{\ep}'(\varphi_{0})(\varphi_{0}-\varphi), v\bigr)_{H} 
       \quad \mbox{for all}\ v \in V. 
\end{align}
Also, we infer from Lemma \ref{ellipticeq2} that 
for all $\theta \in H$ there exists a unique function $\overline{\varphi} \in W$ 
satisfying 
\begin{align}\label{formapB}
&\tau\overline{\varphi} + (1-\Delta)^{-1}\overline{\varphi} 
   - \gamma h\Delta\overline{\varphi} 
   + h\beta_{\ep}(\overline{\varphi}) 
   + h\sigma'(\overline{\varphi})
\\ \notag
&= \tau\varphi_{0} + h(1-\Delta)^{-1}\mu_{0} 
     + (1-\Delta)^{-1}\varphi_{0} 
     + h\lambda_{\ep}'(\varphi_{0})\theta     
         \quad \mbox{a.e.\ on}\ \Omega.    
\end{align}
Thus we can define ${\cal A} : H \to H$, ${\cal B} : H \to H$ 
and ${\cal S} : H \to H$ as 
$$
{\cal A}(\varphi) = \overline{\theta},\ 
{\cal B}(\theta) = \overline{\varphi} 
\quad \mbox{for}\ \varphi, \theta \in H 
$$
and 
$$
{\cal S} = {\cal B} \circ {\cal A},   
$$
respectively. 
We are going to show that, for suitable value of $h$, 
${\cal S}$ is a contraction mapping in $H$. 
Now we let $\varphi, \widetilde{\varphi} \in H$. 
Then, since we can deduce from \eqref{formapA} that 
\begin{align*}
&c_{s}\bigl(\mbox{\rm Ln$_{\ep}$}({\cal A}(\varphi)) 
              - \mbox{\rm Ln$_{\ep}$}({\cal A}(\widetilde{\varphi})), 
                                   {\cal A}(\varphi) - {\cal A}(\widetilde{\varphi}) \bigr)_{H} 
\\ \notag 
&+ \eta h\int_{\Omega}|\nabla({\cal A}(\varphi) - {\cal A}(\widetilde{\varphi}))|^2 
  + h\int_{\Gamma} 
           \alpha_{\Gamma}|{\cal A}(\varphi) - {\cal A}(\widetilde{\varphi})|^2 
\\ \notag 
& = -\bigl(\lambda_{\ep}'(\varphi_{0})(\varphi-\widetilde{\varphi}), 
                                {\cal A}(\varphi) - {\cal A}(\widetilde{\varphi})\bigr)_{H} 
\\ \notag 
&\leq \|\lambda_{\ep}'\|_{L^{\infty}(\mathbb{R})}
                        \|\varphi-\widetilde{\varphi}\|_{H}
                                 \|{\cal A}(\varphi) - {\cal A}(\widetilde{\varphi})\|_{H}, 
\end{align*} 
it follows from Remark \ref{about.Lnep} and (C4) that
\begin{align}\label{esti1forfixthm}
c_{s}\ep\|{\cal A}(\varphi) - {\cal A}(\widetilde{\varphi})\|_{H} 
\leq \|\lambda_{\ep}'\|_{L^{\infty}(\mathbb{R})}\|\varphi-\widetilde{\varphi}\|_{H}. 
\end{align}
Moreover, we have from \eqref{formapB} that 
\begin{align*}
&\tau\|{\cal S}(\varphi)-{\cal S}(\widetilde{\varphi})\|_{H}^2 
+ \bigl( (1-\Delta)^{-1}({\cal S}(\varphi)-{\cal S}(\widetilde{\varphi})), 
                                 {\cal S}(\varphi)-{\cal S}(\widetilde{\varphi}) \bigr)_{H} 
\\ \notag 
&+ \gamma h \|\nabla({\cal S}(\varphi)-{\cal S}(\widetilde{\varphi}))\|_{H}^2 
+ h \bigl( \beta_{\ep}({\cal S}(\varphi))-\beta_{\ep}({\cal S}(\widetilde{\varphi})), 
                                      {\cal S}(\varphi)-{\cal S}(\widetilde{\varphi}) \bigr)_{H} 
\\ \notag 
&+ h \bigl( \sigma'({\cal S}(\varphi))-\sigma'({\cal S}(\widetilde{\varphi})), 
                                      {\cal S}(\varphi)-{\cal S}(\widetilde{\varphi}) \bigr)_{H} 
\\ \notag 
&= h\bigl(\lambda_{\ep}'(\varphi_{0})
                      ({\cal A}(\varphi) - {\cal A}(\widetilde{\varphi})), 
                                {\cal S}(\varphi)-{\cal S}(\widetilde{\varphi})\bigr)_{H} 
\\ \notag 
&\leq h\|\lambda_{\ep}'\|_{L^{\infty}(\mathbb{R})}
               \|{\cal A}(\varphi) - {\cal A}(\widetilde{\varphi})\|_{H}
                                 \|{\cal S}(\varphi)-{\cal S}(\widetilde{\varphi})\|_{H}, 
\end{align*}
and hence combining the monotonicity of $(1-\Delta)^{-1}$ and $\beta_{\ep}$, 
the Lipschitz continuity of $\sigma'$  
with $0 < h < \frac{\tau}{2\|\sigma''\|_{L^{\infty}(\mathbb{R})}}$,  
leads to the inequality  
\begin{align}\label{esti2forfixthm}
\frac{\tau}{2}\|{\cal S}(\varphi)-{\cal S}(\widetilde{\varphi})\|_{H} 
\leq h\|\lambda_{\ep}'\|_{L^{\infty}(\mathbb{R})}
               \|{\cal A}(\varphi) - {\cal A}(\widetilde{\varphi})\|_{H}. 
\end{align}
Therefore we see from \eqref{esti1forfixthm} and \eqref{esti2forfixthm} that 
$$
\|{\cal S}(\varphi)-{\cal S}(\widetilde{\varphi})\|_{H} 
\leq \frac{2\|\lambda_{\ep}'\|_{L^{\infty}(\mathbb{R})}^2}{c_{s}\ep\tau}h
                                                          \|\varphi - \widetilde{\varphi}\|_{H}. 
$$
Then, letting 
$h \in \left(0, 
             \min\left\{\frac{\tau}{2\|\sigma''\|_{L^{\infty}(\mathbb{R})}}, 
              \frac{c_{s}\ep\tau}{2\|\lambda_{\ep}'\|_{L^{\infty}(\mathbb{R})}^2} \right\} 
    \right)$,  
the Banach fixed-point theorem 
allows us to infer that 
there exists a unique function $\varphi_{1} \in H$ 
satisfying $\varphi_{1} = {\cal S}(\varphi_{1}) \in W$. 
Hence, putting 
$$
\theta_{1} := {\cal A}(\varphi_{1}) \in V 
$$
and 
$$
\mu_{1} 
:= (1-\Delta)^{-1}\mu_{0} 
  + \frac{1}{h}(1-\Delta)^{-1}\varphi_{0} 
  - \frac{1}{h}(1-\Delta)^{-1}\varphi_{1},  
$$
we can obtain \eqref{writtendfPtauepnsol1}-\eqref{writtendfPtauepnsol3} 
in the case that $n=0$.    
Thus, by extending the argument to any $n$, we can verify that for all 
$h \in \left(0, 
             \min\left\{\frac{\tau}{2\|\sigma''\|_{L^{\infty}(\mathbb{R})}}, 
              \frac{c_{s}\ep\tau}{2\|\lambda_{\ep}'\|_{L^{\infty}(\mathbb{R})}^2} \right\} 
    \right)$ 
there exists a unique weak solution of {\rm \ref{Ptauepn}} for $n=0, ..., N-1$.   
\qed
\end{prlem3.2}

\vspace{10pt}

%%==============================================================%%
%%==============                                  ==============%%
%%======                      Section4                    ======%%
%%====                                                      ====%%
%%==                                                          ==%%
%%====                                                      ====%%
%%======                                                  ======%%
%%==============                                  ==============%%
%%==============================================================%%

\section{Estimates for \ref{Ptaueph} and passage to the limit as $h\searrow0$} 
\label{Sec4}

In this section we will prove Lemma \ref{existPtauep}. 
We will establish estimates for \ref{Ptaueph} 
to derive existence for \ref{Ptauep}  
by passing to the limit in \ref{Ptaueph} as $h\searrow0$.  
\begin{lem}\label{firstestiPtaueph}
There exist constants $\ep_{1} \in (0, 1]$   
and $C>0$ depending on the data 
such that 
\begin{align*} 
&\ep\|\overline{\theta}_{h}\|_{L^{\infty}(0, T; H)}^2 
  + \ep\|\ln_{\ep}(\overline{\theta}_{h})\|_{L^{\infty}(0, T; H)}^2 
  + \|\rho_{\ep}(\overline{\theta}_{h})\|_{L^{\infty}(0, T; L^1(\Omega))} 
\\ \notag 
&+ \|\overline{\theta}_{h}\|_{L^2(0, T; V)}^2 
  + \|\partial_{t}\hat{\varphi}_{h} 
                        + h\partial_{t}\hat{\mu}_{h}\|_{L^2(0, T; V_{0}^{*})}^2 
  + \tau\|\partial_{t}\hat{\varphi}_{h}\|_{L^2(0, T; H)}^2 
  + \|\overline{\varphi}_{h}\|_{L^{\infty}(0, T; V)}^2 
\\ \notag 
&+ h\|\overline{\mu}_{h}\|_{L^{\infty}(0, T; H)}^2  
  + h^2\|\partial_{t}\hat{\mu}_{h}\|_{L^2(0, T; H)}^2  
  + \|\hat{\beta}_{\ep}(\overline{\varphi}_{h})\|_{L^{\infty}(0, T; L^1(\Omega))} 
\leq C  
\end{align*}
for all $h$ with 
$$ 
0 <h < h_{0} := \min\left\{1, \frac{\tau}{2\|\sigma''\|_{L^{\infty}(\mathbb{R})}}, 
               \frac{c_{s}\ep\tau}{2\|\lambda_{\ep}'\|_{L^{\infty}(\mathbb{R})}^2}, 
                            \frac{\gamma}{8\|\sigma''\|_{L^{\infty}(\mathbb{R})}^2} \right\},    
$$ 
$\ep \in (0, \ep_{1})$ and $\tau>0$.    
\end{lem}
\begin{proof}
Taking $v = h\theta_{n+1}$ in \eqref{dfPtauepnsol1} and using (C4) 
lead to the inequality 
\begin{align}\label{a1}
&c_{s}(\theta_{n+1}, 
          \mbox{\rm Ln$_{\ep}$}(\theta_{n+1})-\mbox{\rm Ln$_{\ep}$}(\theta_{n}))_{H} 
+ \eta h \|\nabla\theta_{n+1}\|_{H}^2 
+ \alpha_{*}h\|\theta_{n+1}\|_{L^2(\Gamma)}^2 
\\ \notag 
&\leq h(f_{n+1}, \theta_{n+1})_{H} 
    + h\int_{\Gamma}\alpha_{\Gamma}\theta_{\Gamma, n+1}\theta_{n+1} 
    - \int_{\Omega}\lambda_{\ep}'(\varphi_{n})(\varphi_{n+1}-\varphi_{n})\theta_{n+1}.  
\end{align}
Here we deduce from Remark \ref{def.of.rhoep} that 
\begin{align}\label{a2}
&c_{s}(\theta_{n+1}, 
          \mbox{\rm Ln$_{\ep}$}(\theta_{n+1})-\mbox{\rm Ln$_{\ep}$}(\theta_{n}))_{H} 
\\ \notag 
&=  c_{s}\ep(\theta_{n+1}, \theta_{n+1}-\theta_{n})_{H} 
+ c_{s}\ep(\ln_{\ep}(\theta_{n+1}), 
                   \ln_{\ep}(\theta_{n+1})-\ln_{\ep}(\theta_{n}))_{H} 
\\ \notag 
&\,\quad + c_{s}\bigl(e^{\ln \rho_{\ep}(\theta_{n+1})}, 
                                     \ln \rho_{\ep}(\theta_{n+1})
                                                -\ln\rho_{\ep}(\theta_{n})\bigr)_{H} 
\\ \notag 
&\geq  \frac{c_{s}\ep}{2}\|\theta_{n+1}\|_{H}^2 
     - \frac{c_{s}\ep}{2}\|\theta_{n}\|_{H}^2 
     + \frac{c_{s}\ep}{2}\|\theta_{n+1}-\theta_{n}\|_{H}^2 
\\ \notag 
&\,\quad+\frac{c_{s}\ep}{2}\|\ln_{\ep}(\theta_{n+1})\|_{H}^2 
     - \frac{c_{s}\ep}{2}\|\ln_{\ep}(\theta_{n})\|_{H}^2 
     + \frac{c_{s}\ep}{2}\|\ln_{\ep}(\theta_{n+1})
                                          -\ln_{\ep}(\theta_{n})\|_{H}^2 
\\ \notag 
&\,\quad + c_{s}\int_{\Omega} \rho_{\ep}(\theta_{n+1}) 
              - c_{s}\int_{\Omega} \rho_{\ep}(\theta_{n}),  
\end{align}
where the inequality $e^x (x-y) \geq e^x - e^y$ ($x, y \in \mathbb{R}$) 
was applied. 
We point out that this inequality holds true as $ze^z \geq e^z -1$ 
for all $z \in \mathbb{R}$. 
Next we observe that the Young inequality, (C4) and (C7) yield 
that there exist constants $C_{1}, C_{2}>0$ satisfying  
\begin{align}\label{a3}
&h(f_{n+1}, \theta_{n+1})_{H} 
    + h\int_{\Gamma}\alpha_{\Gamma}\theta_{\Gamma, n+1}\theta_{n+1} 
\\ \notag 
&\leq C_{1} h\|f_{n+1}\|_{H}^2 
+ \frac{\min\left\{\eta, \frac{\alpha_{*}}{2} \right\}}{2C^{*}}h\|\theta_{n+1}\|_{V}^2 
         + \alpha^{*}\theta^{*}h\int_{\Gamma}|\theta_{n+1}| 
\\ \notag 
&\leq C_{1} h\|f_{n+1}\|_{H}^2 
+ \frac{\min\left\{\eta, \frac{\alpha_{*}}{2} \right\}}{2C^{*}}h\|\theta_{n+1}\|_{V}^2 
         + \frac{\alpha_{*}h}{2}\|\theta_{n+1}\|_{L^2(\Gamma)}^2 + C_{2}h.   
\end{align}
Thus we see from \eqref{a1}-\eqref{a3} and \eqref{keyineq} that 
\begin{align}\label{a4}
&\frac{c_{s}\ep}{2}\|\theta_{n+1}\|_{H}^2 
     - \frac{c_{s}\ep}{2}\|\theta_{n}\|_{H}^2 
     + \frac{c_{s}\ep}{2}\|\theta_{n+1}-\theta_{n}\|_{H}^2 
\\ \notag 
&+\frac{c_{s}\ep}{2}\|\ln_{\ep}(\theta_{n+1})\|_{H}^2 
     - \frac{c_{s}\ep}{2}\|\ln_{\ep}(\theta_{n})\|_{H}^2 
     + \frac{c_{s}\ep}{2}\|\ln_{\ep}(\theta_{n+1})
                                          -\ln_{\ep}(\theta_{n})\|_{H}^2 
\\ \notag 
& + c_{s}\int_{\Omega} \rho_{\ep}(\theta_{n+1}) 
              - c_{s}\int_{\Omega} \rho_{\ep}(\theta_{n}) 
+ \frac{\min\left\{\eta, \frac{\alpha_{*}}{2} \right\}}{2C^{*}}h\|\theta_{n+1}\|_{V}^2 
\\ \notag 
&\leq  C_{1} h\|f_{n+1}\|_{H}^2 + C_{2}h 
          - \int_{\Omega}\lambda_{\ep}'(\varphi_{n})(\varphi_{n+1}-\varphi_{n})  
                                                                                              \theta_{n+1}. 
\end{align}
We sum \eqref{a4} over $n=0, ..., m-1$ with $1 \leq m \leq N$ 
to obtain that 
\begin{align}\label{a5}
&\frac{c_{s}\ep}{2}\|\theta_{m}\|_{H}^2 
  +\frac{c_{s}\ep}{2}\|\ln_{\ep}(\theta_{m})\|_{H}^2      
  + c_{s}\int_{\Omega} \rho_{\ep}(\theta_{m}) 
  + \frac{\min\left\{\eta, \frac{\alpha_{*}}{2} \right\}}{2C^{*}}h
                         \sum_{n=0}^{m-1}\|\theta_{n+1}\|_{V}^2 
\\ \notag 
&\leq \frac{c_{s}\ep}{2}\|\theta_{0}\|_{H}^2 
         + \frac{c_{s}\ep}{2}\|\ln_{\ep}(\theta_{0})\|_{H}^2      
         + c_{s}\int_{\Omega} \rho_{\ep}(\theta_{0}) 
\\ \notag
&\,\quad + C_{1} h\sum_{n=0}^{m-1}\|f_{n+1}\|_{H}^2 
       + C_{2}T 
          - \sum_{n=0}^{m-1}
                 \int_{\Omega}\lambda_{\ep}'(\varphi_{n})(\varphi_{n+1}-\varphi_{n})  
                                                                                             \theta_{n+1}.  
\end{align}
Here, recalling Remark \ref{def.of.rhoep}, we have that 
\begin{align}\label{a6}
&\frac{c_{s}\ep}{2}\|\ln_{\ep}(\theta_{0})\|_{H}^2 
          + c_{s}\int_{\Omega} \rho_{\ep}(\theta_{0}) 
\\ \notag 
&= \frac{c_{s}\ep}{2}\|\ln_{\ep}(\theta_{0})\|_{H}^2  
    + c_{s}\int_{\Omega} (\theta_{0} - \ep\ln_{\ep}(\theta_{0})) 
\\ \notag 
&\leq  \frac{c_{s}}{2}\|\ln \theta_{0}\|_{H}^2 
          + c_{s}\|\theta_{0}\|_{L^1(\Omega)} 
          + c_{s}\|\ln \theta_{0}\|_{L^1(\Omega)} 
\\ \notag 
&\leq  \frac{c_{s}|\Omega|}{2}\max_{\theta_{*}\leq r \leq \theta^{*}}|\ln r|^2 
          + c_{s}\|\theta_{0}\|_{L^1(\Omega)}   
          + c_{s}|\Omega|\max_{\theta_{*}\leq r \leq \theta^{*}}|\ln r|   
\end{align}
for all $\ep \in (0, 1]$. 
Hence, owing to \eqref{a5} and \eqref{a6}, there is a constant $C_3$ such that   
\begin{align}\label{a7}
&\frac{c_{s}\ep}{2}\|\theta_{m}\|_{H}^2 
  +\frac{c_{s}\ep}{2}\|\ln_{\ep}(\theta_{m})\|_{H}^2      
  + c_{s}\int_{\Omega} \rho_{\ep}(\theta_{m}) 
  + \frac{\min\left\{\eta, \frac{\alpha_{*}}{2} \right\}}{2C^{*}}h
                         \sum_{n=0}^{m-1}\|\theta_{n+1}\|_{V}^2 
\\ \notag 
&\leq C_{3}
         - \sum_{n=0}^{m-1}
                 \int_{\Omega}\lambda_{\ep}'(\varphi_{n})(\varphi_{n+1}-\varphi_{n})  
                                                                                                \theta_{n+1} 
\end{align} 
for all $h \in (0, h_{0})$, $\ep \in (0, 1]$, $\tau > 0$, for $m=1, ..., N$. 
It follows from \eqref{dfPtauepnsol2} that   
\begin{align}\label{tauepnint0}
\int_{\Omega} \left(\frac{\varphi_{n+1}-\varphi_{n}}{h} 
                                    + \mu_{n+1} - \mu_{n} \right) 
= 0
\end{align}
for $n=0, ... , N-1$. 
Using \eqref{innerVzerostar},   
multiplying \eqref{dfPtauepnsol2} by 
$h{\cal N}\left(\frac{\varphi_{n+1}-\varphi_{n}}{h} 
                                            + \mu_{n+1} - \mu_{n}\right)$ 
and integrating over $\Omega$ 
yield that 
\begin{align}\label{a8}
&h\left\|\frac{\varphi_{n+1}-\varphi_{n}}{h} 
                           + \mu_{n+1} - \mu_{n}\right\|_{V_{0}^{*}}^2 
\\ \notag 
&+ h\int_{\Omega}\nabla \mu_{n+1} \cdot 
      \nabla {\cal N}\left(\frac{\varphi_{n+1}-\varphi_{n}}{h} 
                                            + \mu_{n+1} - \mu_{n}\right) 
= 0. 
\end{align}
Here we infer from \eqref{defN} and \eqref{tauepnint0} that 
\begin{align}\label{a9}
&h\int_{\Omega}\nabla \mu_{n+1} \cdot 
      \nabla {\cal N}\left(\frac{\varphi_{n+1}-\varphi_{n}}{h} 
                                            + \mu_{n+1} - \mu_{n}\right) 
\\ \notag 
&= h\int_{\Omega}
         \nabla \left(\mu_{n+1}
                     -\frac{1}{|\Omega|}\int_{\Omega}\mu_{n+1} \right) \cdot 
      \nabla {\cal N}\left(\frac{\varphi_{n+1}-\varphi_{n}}{h} 
                                            + \mu_{n+1} - \mu_{n}\right) 
\\ \notag 
&=h \left\langle 
         \frac{\varphi_{n+1}-\varphi_{n}}{h} 
                                            + \mu_{n+1} - \mu_{n}, 
            \mu_{n+1}
                     -\frac{1}{|\Omega|}\int_{\Omega}\mu_{n+1} 
       \right\rangle_{V_{0}^{*}, V_{0}}
\\ \notag 
&=h \left( 
         \frac{\varphi_{n+1}-\varphi_{n}}{h} 
                                            + \mu_{n+1} - \mu_{n}, 
            \mu_{n+1}
                     -\frac{1}{|\Omega|}\int_{\Omega}\mu_{n+1} 
       \right)_{H} 
\\ \notag 
&= h \left( 
         \frac{\varphi_{n+1}-\varphi_{n}}{h} 
                                            + \mu_{n+1} - \mu_{n}, 
            \mu_{n+1} 
       \right)_{H}.  
\end{align}
Thus we derive from \eqref{a8}, \eqref{a9} and \eqref{dfPtauepnsol3} that  
\begin{align}\label{a10}
&h\left\|\frac{\varphi_{n+1}-\varphi_{n}}{h} 
                           + \mu_{n+1} - \mu_{n}\right\|_{V_{0}^{*}}^2 
+ \tau h\left\|\frac{\varphi_{n+1}-\varphi_{n}}{h}\right\|_{H}^2 
\\ \notag 
&+ \gamma(\nabla\varphi_{n+1}, 
                        \nabla(\varphi_{n+1}-\varphi_{n}))_{H} 
+ (\beta_{\ep}(\varphi_{n+1}), 
                                         \varphi_{n+1}-\varphi_{n})_{H} 
\\ \notag 
&+  h(\mu_{n+1}-\mu_{n}, \mu_{n+1})_{H} 
\\ \notag 
&= - h\left(\sigma'(\varphi_{n+1}), 
                       \frac{\varphi_{n+1}-\varphi_{n}}{h} \right)_{H} 
    + \int_{\Omega}\lambda_{\ep}'(\varphi_{n})(\varphi_{n+1}-\varphi_{n})\theta_{n+1}.  
\end{align}
On the other hand, we have that 
\begin{align}\label{a11}
&\gamma(\nabla\varphi_{n+1}, 
                        \nabla(\varphi_{n+1}-\varphi_{n}))_{H} 
\\ \notag 
&= \frac{\gamma}{2}\|\nabla\varphi_{n+1}\|_{H}^2 
     - \frac{\gamma}{2}\|\nabla\varphi_{n}\|_{H}^2 
     + \frac{\gamma}{2}\|\nabla(\varphi_{n+1}
                                                                 -\varphi_{n})\|_{H}^2 
\end{align}
and 
\begin{align}\label{a12}
& h(\mu_{n+1}-\mu_{n}, \mu_{n+1})_{H} 
\\ \notag 
&= \frac{h}{2}\|\mu_{n+1}\|_{H}^2 
     - \frac{h}{2}\|\mu_{n}\|_{H}^2 
     + \frac{h}{2}\|\mu_{n+1}-\mu_{n}\|_{H}^2.  
\end{align}
By Remark \ref{aboutbetahatep} and the definition of subdifferential, 
it holds that 
\begin{align}\label{a13}
(\beta_{\ep}(\varphi_{n+1}), 
                                         \varphi_{n+1}-\varphi_{n})_{H} 
\geq \int_{\Omega}\hat{\beta}_{\ep}(\varphi_{n+1}) 
        - \int_{\Omega}\hat{\beta}_{\ep}(\varphi_{n}).  
\end{align}
We see from \eqref{tauepnint0} and the Young inequality that 
\begin{align}\label{a14}
& - h\left(\sigma'(\varphi_{n+1}), 
                       \frac{\varphi_{n+1}-\varphi_{n}}{h} \right)_{H} 
\\ \notag 
&= - h\left\langle
            \frac{\varphi_{n+1}-\varphi_{n}}{h} + \mu_{n+1} - \mu_{n}, 
                \sigma'(\varphi_{n+1})
                   -\frac{1}{|\Omega|}\int_{\Omega}\sigma'(\varphi_{n+1})   
         \right\rangle_{V_{0}^{*}, V_{0}} 
\\ \notag 
 &\,\quad+ h\left(\sigma'(\varphi_{n+1}), \mu_{n+1} - \mu_{n} \right)_{H} 
\\ \notag 
&\leq \frac{h}{2}\left\|\frac{\varphi_{n+1}-\varphi_{n}}{h} 
                                          + \mu_{n+1} - \mu_{n}\right\|_{V_{0}^{*}}^2 
        + \frac{\|\sigma''\|_{L^{\infty}(\mathbb{R})}^2 h}{2}
                                                               \|\nabla \varphi_{n+1}\|_{H}^2 
\\ \notag 
&\,\quad + 2\|\sigma''\|_{L^{\infty}(\mathbb{R})}^2 h\|\varphi_{n+1}\|_{H}^2 
             + 2|\sigma'(0)|^2 |\Omega| h  
             + \frac{h}{4}\|\mu_{n+1}-\mu_{n}\|_{H}^2. 
\end{align}
Hence it follows from \eqref{a10}-\eqref{a14} that   
\begin{align}\label{a15}
&\frac{h}{2}\left\|\frac{\varphi_{n+1}-\varphi_{n}}{h} 
                           + \mu_{n+1} - \mu_{n}\right\|_{V_{0}^{*}}^2 
+ \tau h
        \left\|\frac{\varphi_{n+1}-\varphi_{n}}{h}\right\|_{H}^2 
\\ \notag 
&+\frac{\gamma}{2}\|\nabla\varphi_{n+1}\|_{H}^2 
     - \frac{\gamma}{2}\|\nabla\varphi_{n}\|_{H}^2 
     + \frac{\gamma}{2}\|\nabla(\varphi_{n+1}
                                                                 -\varphi_{n})\|_{H}^2 
\\ \notag 
&+\int_{\Omega}\hat{\beta}_{\ep}(\varphi_{n+1}) 
                  - \int_{\Omega}\hat{\beta}_{\ep}(\varphi_{n}) 
+ \frac{h}{2}\|\mu_{n+1}\|_{H}^2 
     - \frac{h}{2}\|\mu_{n}\|_{H}^2 
+ \frac{h}{4}\|\mu_{n+1}-\mu_{n}\|_{H}^2 
\\ \notag 
&\leq  2\|\sigma''\|_{L^{\infty}(\mathbb{R})}^2 h\|\varphi_{n+1}\|_{V}^2 
             + 2|\sigma'(0)|^2 |\Omega| h  
    + \int_{\Omega}\lambda_{\ep}'(\varphi_{n})(\varphi_{n+1}-\varphi_{n})\theta_{n+1}.
\end{align}
Therefore summing \eqref{a15} over $n=0, ..., m-1$ with $1 \leq m \leq N$ 
and using Remark \ref{aboutbetahatep}  
lead to the identity 
\begin{align}\label{a16}
&\frac{h}{2}\sum_{n=0}^{m-1}\left\|\frac{\varphi_{n+1}-\varphi_{n}}{h} 
                           + \mu_{n+1} - \mu_{n}\right\|_{V_{0}^{*}}^2 
+ \tau h\sum_{n=0}^{m-1}
        \left\|\frac{\varphi_{n+1}-\varphi_{n}}{h}\right\|_{H}^2 
\\ \notag 
&+\frac{\gamma}{2}\|\nabla\varphi_{m}\|_{H}^2 
+\int_{\Omega}\hat{\beta}_{\ep}(\varphi_{m}) 
+ \frac{h}{2}\|\mu_{m}\|_{H}^2 
+ \frac{h}{4}\sum_{n=0}^{m-1}\|\mu_{n+1}-\mu_{n}\|_{H}^2 
\\ \notag 
&\leq \frac{\gamma}{2}\|\nabla\varphi_{0}\|_{H}^2 
         + \int_{\Omega}\hat{\beta}(\varphi_{0}) 
         %+ \frac{h}{2}\|\mu_{0}\|_{H}^2 
        + 2\|\sigma''\|_{L^{\infty}(\mathbb{R})}^2 h
                                           \sum_{n=0}^{m-1}\|\varphi_{n+1}\|_{V}^2 
     \\ \notag  
&\,\quad+ 2|\Omega||\sigma'(0)|^2 T 
     + \sum_{n=0}^{m-1} 
       \int_{\Omega}\lambda_{\ep}'(\varphi_{n})(\varphi_{n+1}-\varphi_{n})\theta_{n+1}. 
\end{align}
Owing to (C2), there exist constants $\ep_{1} \in (0, 1]$ and $C_{4} > 0$ 
such that 
\begin{align}\label{a17} 
\frac{1}{2}\hat{\beta}_{\ep}(r) 
\geq \frac{\gamma}{2} r^2 - C_{4}
\end{align}
for all $r \in \mathbb{R}$ and all $\ep \in (0, \ep_{1})$ 
(see, e.g., \cite[Lemma 4.1]{BCFG2007}). 
Thus we deduce from \eqref{a16} and \eqref{a17} that 
\begin{align*}%\label{a18}
&\frac{h}{2}\sum_{n=0}^{m-1}\left\|\frac{\varphi_{n+1}-\varphi_{n}}{h} 
                           + \mu_{n+1} - \mu_{n}\right\|_{V_{0}^{*}}^2 
+ \tau h\sum_{n=0}^{m-1}
        \left\|\frac{\varphi_{n+1}-\varphi_{n}}{h}\right\|_{H}^2 
\\ \notag 
&+\left(\frac{\gamma}{2}-2\|\sigma''\|_{L^{\infty}(\mathbb{R})}^2 h \right)
                                                                                   \|\varphi_{m}\|_{V}^2 
+ \frac{1}{2}\int_{\Omega}\hat{\beta}_{\ep}(\varphi_{m}) 
+ \frac{h}{2}\|\mu_{m}\|_{H}^2 
+ \frac{h}{4}\sum_{n=0}^{m-1}\|\mu_{n+1}-\mu_{n}\|_{H}^2 
\\ \notag 
&\leq \frac{\gamma}{2}\|\nabla\varphi_{0}\|_{H}^2 
         + \int_{\Omega}\hat{\beta}(\varphi_{0}) 
         %+ \frac{h}{2}\|\mu_{0}\|_{H}^2 
        + 2\|\sigma''\|_{L^{\infty}(\mathbb{R})}^2 h
                                           \sum_{n=0}^{m-2}\|\varphi_{n+1}\|_{V}^2 
     \\ \notag  
&\,\quad+ 2|\Omega||\sigma'(0)|^2 T + C_{4}|\Omega|
     + \sum_{n=0}^{m-1} 
       \int_{\Omega}\lambda_{\ep}'(\varphi_{n})(\varphi_{n+1}-\varphi_{n})\theta_{n+1} 
\end{align*} 
for all $h \in (0, h_{0})$, $\ep \in (0, \ep_{1})$, $\tau > 0$, for $m=1, ..., N$. 
Hence there exists a constant $C_{5} > 0$ such that 
\begin{align}\label{a19}
&\frac{h}{2}\sum_{n=0}^{m-1}\left\|\frac{\varphi_{n+1}-\varphi_{n}}{h} 
                           + \mu_{n+1} - \mu_{n}\right\|_{V_{0}^{*}}^2 
+ \tau h\sum_{n=0}^{m-1}
        \left\|\frac{\varphi_{n+1}-\varphi_{n}}{h}\right\|_{H}^2 
\\ \notag 
&+ \frac{\gamma}{4}\|\varphi_{m}\|_{V}^2 
+ \frac{1}{2}\int_{\Omega}\hat{\beta}_{\ep}(\varphi_{m}) 
+ \frac{h}{2}\|\mu_{m}\|_{H}^2 
+ \frac{h}{4}\sum_{n=0}^{m-1}\|\mu_{n+1}-\mu_{n}\|_{H}^2 
\\ \notag 
&\leq C_{5}
        + 2\|\sigma''\|_{L^{\infty}(\mathbb{R})}^2 h
                                           \sum_{n=0}^{m-2}\|\varphi_{n+1}\|_{V}^2 
     + \sum_{n=0}^{m-1} 
       \int_{\Omega}\lambda_{\ep}'(\varphi_{n})(\varphi_{n+1}-\varphi_{n})\theta_{n+1}
\end{align}
for all $h \in (0, h_{0})$, $\ep \in (0, \ep_{1})$, $\tau > 0$, for $m=1, ..., N$.    
Thus we combine \eqref{a7} and \eqref{a19} to derive that 
\begin{align*}%\label{a20}
&\frac{c_{s}\ep}{2}\|\theta_{m}\|_{H}^2 
  +\frac{c_{s}\ep}{2}\|\ln_{\ep}(\theta_{m})\|_{H}^2      
  + c_{s}\int_{\Omega} \rho_{\ep}(\theta_{m}) 
  + \frac{\min\left\{\eta, \frac{\alpha_{*}}{2} \right\}}{2C^{*}}h
                         \sum_{n=0}^{m-1}\|\theta_{n+1}\|_{V}^2 
\\ \notag 
& + \frac{h}{2}\sum_{n=0}^{m-1}\left\|\frac{\varphi_{n+1}-\varphi_{n}}{h} 
                           + \mu_{n+1} - \mu_{n}\right\|_{V_{0}^{*}}^2 
+ \tau h\sum_{n=0}^{m-1}
        \left\|\frac{\varphi_{n+1}-\varphi_{n}}{h}\right\|_{H}^2 
\\ \notag 
&+ \frac{\gamma}{4}\|\varphi_{m}\|_{V}^2 
+ \frac{1}{2}\int_{\Omega}\hat{\beta}_{\ep}(\varphi_{m}) 
+ \frac{h}{2}\|\mu_{m}\|_{H}^2 
+ \frac{h}{4}\sum_{n=0}^{m-1}\|\mu_{n+1}-\mu_{n}\|_{H}^2 
\\ \notag 
&\leq C_{3} +C_{5} 
        + 2\|\sigma''\|_{L^{\infty}(\mathbb{R})}^2 h
                                           \sum_{n=0}^{m-2}\|\varphi_{n+1}\|_{V}^2 
\end{align*}
for all $h \in (0, h_{0})$, $\ep \in (0, \ep_{1})$, $\tau > 0$, for $m=1, ..., N$. 
Therefore, by virtue of the discrete Gronwall lemma 
(see, e.g., \cite[Prop.\ 2.2.1]{Jerome}), 
there exists a constant $C_{6} > 0$ such that 
\begin{align*}
&\ep\|\theta_{m}\|_{H}^2 
  + \ep\|\ln_{\ep}(\theta_{m})\|_{H}^2      
  + \int_{\Omega} \rho_{\ep}(\theta_{m}) 
  + h\sum_{n=0}^{m-1}\|\theta_{n+1}\|_{V}^2 
\\ \notag 
& + h\sum_{n=0}^{m-1}\left\|\frac{\varphi_{n+1}-\varphi_{n}}{h} 
                                    + \mu_{n+1} - \mu_{n}\right\|_{V_{0}^{*}}^2 
+ \tau h\sum_{n=0}^{m-1}
        \left\|\frac{\varphi_{n+1}-\varphi_{n}}{h}\right\|_{H}^2 
\\ \notag 
&+\|\varphi_{m}\|_{V}^2 
+ \int_{\Omega}\hat{\beta}_{\ep}(\varphi_{m}) 
+ h\|\mu_{m}\|_{H}^2 
+ h\sum_{n=0}^{m-1}\|\mu_{n+1}-\mu_{n}\|_{H}^2 
\\ \notag 
&\leq C_{6}
\end{align*}  
for all $h \in (0, h_{0})$, $\ep \in (0, \ep_{1})$, $\tau > 0$, for $m=1, ..., N$,  
which means that Lemma \ref{firstestiPtaueph} holds 
by \eqref{delta1}, \eqref{hat2}-\eqref{line2}.       
\end{proof}

\begin{lem}\label{secondestiPtaueph}
Let $h_{0}$, $\ep_{1}$ be as in Lemma \ref{firstestiPtaueph} 
and let $\overline{\tau}$ be as in Theorem \ref{maintheorem1}.   
Then there exists a constant $C>0$ depending on the data such that  
$$
\|\overline{\mu}_{h}\|_{L^2(0, T; V)}^2 \leq C 
$$
for all $h \in (0, h_{0})$, $\ep \in (0, \ep_{1})$ and $\tau \in (0, \overline{\tau})$.   
\end{lem}
\begin{proof}
Since we can obtain the identity 
\begin{align}\label{b1}
h\|\nabla\mu_{n+1}\|_{H}^2 
= -h \left(\frac{\varphi_{n+1}-\varphi_{n}}{h}+\mu_{n+1}-\mu_{n}, 
                                                                              \mu_{n+1} \right)_{H} 
\end{align}
by multiplying \eqref{dfPtauepnsol2} by $h\mu_{n+1}$ 
and integrating over $\Omega$, 
we deduce from combining \eqref{a8}, \eqref{a9} and \eqref{b1} that 
\begin{align}\label{b2}
h\|\nabla\mu_{n+1}\|_{H}^2 
= h \left\|\frac{\varphi_{n+1}-\varphi_{n}}{h}+\mu_{n+1}-\mu_{n} \right\|_{V_{0}^{*}}^2.  
\end{align}
Thus \eqref{b2} and Lemma \ref{firstestiPtaueph} imply that 
there exists a constant $C_{1} > 0$ such that 
\begin{align}\label{b3}
\|\nabla\overline{\mu}_{h}\|_{L^2(0, T; H)}^2 
= \left\|\partial_{t}\hat{\varphi}_{h}
                                     +h\partial_{t}\hat{\mu}_{h} \right\|_{L^2(0, T; V_{0}^{*})}^2 
\leq C_{1} 
\end{align}
for all $h \in (0, h_{0})$, $\ep \in (0, \ep_{1})$ and $\tau > 0$.   
It follows from \eqref{tauepnint0} that 
\begin{align*}
\int_{\Omega} (\varphi_{n+1}+h\mu_{n+1}) 
= \int_{\Omega} (\varphi_{n}+h\mu_{n}) 
\end{align*}
for $n = 0, ..., N-1$, and consequently \eqref{muzerozero} enables us to infer that 
\begin{align}\label{int.varphiplushmu}
\frac{1}{|\Omega|}\int_{\Omega} (\varphi_{j}+h\mu_{j}) 
= \frac{1}{|\Omega|}\int_{\Omega} \varphi_{0} 
= m_{0}  
\end{align}
for $j = 0, 1, ..., N$.   
Multiplying \eqref{dfPtauepnsol3} by $\varphi_{n+1} + h\mu_{n+1} - m_{0}$ 
and integrating over $\Omega$ lead to the identity 
\begin{align}\label{b4}
&h\|\mu_{n+1}\|_{H}^2 + \gamma \|\nabla \varphi_{n+1}\|_{H}^2 
+ (\beta_{\ep}(\varphi_{n+1}), \varphi_{n+1}-m_{0})_{H} 
\\ \notag 
&= (\varphi_{n+1} + h\mu_{n+1} - m_{0}, \mu_{n+1})_{H} 
      - \left(\tau\frac{\varphi_{n+1}-\varphi_{n}}{h}, 
                                                              \varphi_{n+1}-m_{0} \right)_{H} 
      \\ \notag 
&\,\quad - (\sigma'(\varphi_{n+1}), \varphi_{n+1}-m_{0})_{H} 
     + (\lambda_{\ep}'(\varphi_{n})\theta_{n+1}, \varphi_{n+1}-m_{0})_{H}. 
\end{align}
Here there exist constant $c>0$ and $d>0$ such that 
\begin{align}\label{b5}
(\beta_{\ep}(\varphi_{n+1}), \varphi_{n+1}-m_{0})_{H} 
\geq c\|\beta_{\ep}(\varphi_{n+1})\|_{L^1(\Omega)} - d   
\end{align}
for all $\ep \in (0, 1]$ 
and for $n = 0, ..., N-1$  
(see, e.g., \cite[Section 5, p.\ 908]{GMS}).  
We can verify from \eqref{int.varphiplushmu} that 
\begin{align}\label{b6}
&(\varphi_{n+1} + h\mu_{n+1} - m_{0}, \mu_{n+1})_{H} 
\\ \notag 
&= \Bigl(\varphi_{n+1} + h\mu_{n+1} - m_{0}, 
               \mu_{n+1} - \frac{1}{|\Omega|}\int_{\Omega}\mu_{n+1} \Bigr)_{H} 
\\ \notag 
&= \Bigl\langle 
          \varphi_{n+1} + h\mu_{n+1} - m_{0}, 
               \mu_{n+1} - \frac{1}{|\Omega|}\int_{\Omega}\mu_{n+1} 
       \Bigr\rangle_{V_{0}^{*}, V_{0}} 
\\ \notag  
& \leq \|\varphi_{n+1} + h\mu_{n+1} - m_{0}\|_{V_{0}^{*}}
            \Bigl\|\mu_{n+1} - \frac{1}{|\Omega|}\int_{\Omega}\mu_{n+1}\Bigr\|_{V_{0}} 
\\ \notag 
&\leq \|\overline{\varphi}_{h} + h\overline{\mu}_{H}-m_{0}\|_{L^{\infty}(0, T; H)}
                                                                               \|\nabla \mu_{n+1}\|_{H}. 
\end{align}
By the Schwarz inequality we have that  
\begin{align}\label{b7}
- \left(\tau\frac{\varphi_{n+1}-\varphi_{n}}{h}, \varphi_{n+1}-m_{0} \right)_{H} 
\leq \overline{\tau}^{1/2}\tau^{1/2}\left\|\frac{\varphi_{n+1}-\varphi_{n}}{h}\right\|_{H}
                                            \|\overline{\varphi}_{h}-m_{0} \|_{L^{\infty}(0, T; H)}
\end{align}
for all $\tau \in (0, \overline{\tau})$.   
We derive from the Lipschitz continuity of $\sigma'$ that 
there exists a constant $C_{2} > 0$ satisfying 
\begin{align}\label{b8}
- (\sigma'(\varphi_{n+1}), \varphi_{n+1}-m_{0})_{H} 
\leq C_{2}(\|\overline{\varphi}_{h}\|_{L^{\infty}(0, T; H)}+1)
                                 \|\overline{\varphi}_{h}-m_{0} \|_{L^{\infty}(0, T; H)}. 
\end{align}
The continuity of the embedding $V \hookrightarrow L^3(\Omega)$ yields that 
there exists a constant $C_{3}>0$ fulfilling  
\begin{align}\label{b9}
&(\lambda_{\ep}'(\varphi_{n})\theta_{n+1}, \varphi_{n+1}-m_{0})_{H} 
\\ \notag 
&\leq \|\lambda_{\ep}'(\underline{\varphi}_{h})\|_{L^{\infty}(0, T; L^6(\Omega))}
             \|\theta_{n+1}\|_{H}
                      \|\overline{\varphi}_{h} - m_{0}\|_{L^{\infty}(0, T; L^3(\Omega))} 
\\ \notag 
&\leq C_{3}\|\lambda_{\ep}'(\underline{\varphi}_{h})\|_{L^{\infty}(0, T; L^6(\Omega))}
       \|\theta_{n+1}\|_{H}\|\overline{\varphi}_{h} - m_{0}\|_{L^{\infty}(0, T; V)}.  
\end{align}
On the other hand, by \eqref{lamep1} and \eqref{lamep2} 
there exists a constant $C_{\lambda} > 0$ such that 
\begin{align*}
|\lambda_{\ep}'(r)| \leq C_{\lambda}(1 + |r|)
\end{align*}
for all $\ep \in (0, 1]$ and all $r \in \mathbb{R}$, 
whence we infer from 
the continuity of the embedding $V \hookrightarrow L^6(\Omega)$ 
and Lemma \ref{firstestiPtaueph} that 
\begin{align}\label{b10}
\|\lambda_{\ep}'(\underline{\varphi}_{h})\|_{L^{\infty}(0, T; L^6(\Omega))} 
\leq \overline{C}
\end{align}
for all $h \in (0, h_{0})$, $\ep \in (0, \ep_{1})$, $\tau>0$ and 
for some constant $\overline{C} > 0$. 
Therefore, 
combining \eqref{b3}, \eqref{b4}-\eqref{b10} and Lemma \ref{firstestiPtaueph}, 
we can deduce that 
there exists a constant $C_{4} > 0$ satisfying 
\begin{align}\label{b11}
\|\beta_{\ep}(\overline{\varphi}_{h})\|_{L^2(0, T; L^1(\Omega))} \leq C_{4}
\end{align}
for all $h \in (0, h_{0})$, $\ep \in (0, \ep_{1})$ and $\tau \in (0, \overline{\tau})$.   
Next integrating \eqref{dfPtauepnsol3} over $\Omega$ leads to the identity  
\begin{align}\label{b12}
\int_{\Omega}\overline{\mu}_{h}(t) 
=  \tau\int_{\Omega}\partial_{t} \hat{\varphi}_{h} (t) 
   + \int_{\Omega}\beta_{\ep}(\overline{\varphi}_{h}(t)) 
   + \int_{\Omega}\sigma'(\overline{\varphi}_{h}(t)) 
   - \int_{\Omega}\lambda_{\ep}'(\underline{\varphi}_{h}(t))\overline{\theta}_{h}(t).   
\end{align}
From \eqref{dfPtauepnsol2} we have  
\begin{align}\label{b13}
\tau\int_{\Omega}\partial_{t} \hat{\varphi}_{h} (t) 
= -h\int_{\Omega}\partial_{t} \hat{\mu}_{h} (t).  
\end{align}
It follows from \eqref{b10} that 
\begin{align}\label{b14}
\left|
- \int_{\Omega}\lambda_{\ep}'(\underline{\varphi}_{h}(t))\overline{\theta}_{h}(t)
\right| 
\leq \|\lambda_{\ep}'(\underline{\varphi}_{h})\|_{L^{\infty}(0, T; L^6(\Omega))}
                                   \|\overline{\theta}_{h}(t)\|_{L^{6/5}(\Omega)} 
\leq C_{5}\|\overline{\theta}_{h}(t)\|_{H} 
\end{align}
for all $h \in (0, h_{0})$, $\ep \in (0, \ep_{1})$ 
and for some constant $C_{5} > 0$. 
Thus we can conclude the proof of Lemma \ref{secondestiPtaueph} 
by virtue of \eqref{b3}, \eqref{b11}-\eqref{b14}, Lemma \ref{firstestiPtaueph} 
and the Poincar\'e--Wirtinger inequality. 
\end{proof}

\begin{lem}\label{thirdestiPtaueph}
Let $h_{0}$, $\ep_{1}$ be as in Lemma \ref{firstestiPtaueph}.   
Then there exists a constant $C>0$ depending on the data such that  
$$
\|\partial_{t}\hat{\varphi}_{h}\|_{L^2(0, T; V^{*})}^2 \leq C 
$$
for all $h \in (0, h_{0})$, $\ep \in (0, \ep_{1})$ and $\tau > 0$.     
\end{lem}
\begin{proof}
We multiply \eqref{dfPtauepnsol2} by $F^{-1}\partial_{t}\hat{\varphi}_{h}(t)$, 
integrate over $\Omega$ and recall \eqref{defF} and \eqref{innerVstar} 
to infer that 
\begin{align*}%\label{kuri2}
\|\partial_{t}\hat{\varphi}_{h}(t)\|_{V^{*}}^2 
&= - h(\partial_{t}\hat{\mu}_{h}(t), F^{-1}\partial_{t}\hat{\varphi}_{h}(t))_{H} 
     -\int_{\Omega} \nabla\overline{\mu}_{h}(t) 
                              \cdot \nabla F^{-1}\partial_{t}\hat{\varphi}_{h}(t)
\\ \notag 
& \leq  h^2\|\partial_{t}\hat{\mu}_{h}(t)\|_{H}^2
          + \|\nabla\overline{\mu}_{h}(t)\|_{H}^2 
          + \frac{1}{2}\|F^{-1}\partial_{t}\hat{\varphi}_{h}(t)\|_{V}^2 
\\ \notag 
&=     h^2\|\partial_{t}\hat{\mu}_{h}(t)\|_{H}^2
         +\|\nabla\overline{\mu}_{h}(t) \|_{H}^2 
         + \frac{1}{2}\|\partial_{t}\hat{\varphi}_{h}(t)\|_{V^{*}}^2. 
\end{align*}
Hence we can conclude that Lemma \ref{thirdestiPtaueph} holds by 
Lemma \ref{firstestiPtaueph} and \eqref{b3}.   
\end{proof}

\begin{lem}\label{4estiPtaueph}
Let $h_{0}$, $\ep_{1}$ be as in Lemma \ref{firstestiPtaueph} 
and let $\overline{\tau}$ be as in Theorem \ref{maintheorem1}.   
Then there exists a constant $C>0$ depending on the data such that  
\begin{align*}
\|\beta_{\ep}(\overline{\varphi}_{h})\|_{L^2(0, T; H)}^2 
\leq C
\end{align*}
for all $h \in (0, h_{0})$, $\ep \in (0, \ep_{1})$ and $\tau \in (0, \overline{\tau})$.  
\end{lem}
\begin{proof}
We see from \eqref{dfPtauepnsol3} that 
\begin{align}\label{kuku1}
&\|\beta_{\ep}(\overline{\varphi}_{h} (t))\|_{H}^2 
\\ \notag 
&=\bigl(\beta_{\ep}(\overline{\varphi}_{h} (t)), 
             \overline{\mu}_{h} (t)-\tau\partial_{t}\hat{\varphi}_{h} (t)
                                               -\sigma'(\overline{\varphi}_{h} (t)) \bigr)_{H} 
\\ \notag 
&\,\quad -\gamma\int_{\Omega}\beta_{\ep}'(\overline{\varphi}_{h} (t))
                                                        |\nabla \overline{\varphi}_{h} (t)|^2 
    + \int_{\Omega}\beta_{\ep}(\overline{\varphi}_{h} (t))
                   \lambda_{\ep}'(\underline{\varphi}_{h} (t))\overline{\theta}_{h} (t) 
\\ \notag 
&\leq \frac{3}{8}\|\beta_{\ep}(\overline{\varphi}_{h} (t))\|_{H}^2 
         + 2(\|\overline{\mu}_{h} (t)\|_{H}^2 
                + \overline{\tau}\tau\|\partial_{t}\hat{\varphi}_{h} (t)\|_{H}^2   
                 + \|\sigma'(\overline{\varphi}_{h} (t))\|_{H}^2) 
\\ \notag 
&\,\quad + \|\beta_{\ep}(\overline{\varphi}_{h} (t))\|_{H}
                     \|\lambda_{\ep}'(\underline{\varphi}_{h} (t))\|_{L^6(\Omega)} 
                                                    \|\overline{\theta}_{h} (t)\|_{L^3(\Omega)}
\end{align} 
for all $h \in (0, h_{0})$, $\ep \in (0, \ep_{1})$ and $\tau \in (0, \overline{\tau})$.   
Here the continuity of the embedding $V \hookrightarrow L^3(\Omega)$ 
yields that 
\begin{align}\label{kuku2}
&\|\beta_{\ep}(\overline{\varphi}_{h} (t))\|_{H}
                     \|\lambda_{\ep}'(\underline{\varphi}_{h} (t))\|_{L^6(\Omega)} 
                                                    \|\overline{\theta}_{h} (t)\|_{L^3(\Omega)} 
\\ \notag 
&\leq C_{1}\|\beta_{\ep}(\overline{\varphi}_{h} (t))\|_{H}
          \|\lambda_{\ep}'(\underline{\varphi}_{h})\|_{L^{\infty}(0, T; L^6(\Omega))} 
                                                                        \|\overline{\theta}_{h} (t)\|_{V} 
\\ \notag 
&\leq \frac{1}{8}\|\beta_{\ep}(\overline{\varphi}_{h} (t))\|_{H}^2 
         + 2C_{1}^2\|\lambda_{\ep}'(\underline{\varphi}_{h})
                                        \|_{L^{\infty}(0, T; L^6(\Omega))}^2 
                                                                    \|\overline{\theta}_{h} (t)\|_{V}^2 
\end{align}
for all $h \in (0, h_{0})$, $\ep \in (0, \ep_{1})$, 
a.a.\ $t \in (0, T)$, and for some constant $C_{1} > 0$. 
Hence combining \eqref{b10}, \eqref{kuku1}, \eqref{kuku2}, 
Lemmas \ref{firstestiPtaueph} and \ref{secondestiPtaueph} 
leads to Lemma \ref{4estiPtaueph}.  
\end{proof}

\begin{lem}\label{5estiPtaueph}
Let $h_{0}$, $\ep_{1}$ be as in Lemma \ref{firstestiPtaueph} 
and let $\overline{\tau}$ be as in Theorem \ref{maintheorem1}.   
Then there exists a constant $C>0$ depending on the data such that  
$$
\|\overline{\varphi}_{h}\|_{L^2(0, T; W)}^2 \leq C 
$$
for all $h \in (0, h_{0})$, $\ep \in (0, \ep_{1})$ and $\tau \in (0, \overline{\tau})$.  
\end{lem}
\begin{proof}
We derive from \eqref{dfPtauepnsol3} that 
\begin{align}\label{pipi1}
\gamma\|\Delta\overline{\varphi}_{h}(t)\|_{H} 
&\leq \|\overline{\mu}_{h}(t)\|_{H} 
         + \tau\|\partial_{t}\hat{\varphi}_{h}(t)\|_{H} 
         + \|\beta_{\ep}(\overline{\varphi}_{h}(t))\|_{H} 
\\ \notag 
&\,\quad + \|\sigma'(\overline{\varphi}_{h}(t))\|_{H} 
         + \|\lambda_{\ep}'(\underline{\varphi}_{h}(t))\overline{\theta}_{h}(t)\|_{H}. 
\end{align}
On the other hand, 
by the continuity of the embedding $V \hookrightarrow L^3(\Omega)$ 
there exists a constant $C_{1} > 0$ such that 
\begin{align}\label{pipi2}
\|\lambda_{\ep}'(\underline{\varphi}_{h}(t))\overline{\theta}_{h}(t)\|_{H}
&\leq \|\lambda_{\ep}'(\underline{\varphi}_{h})\|_{L^{\infty}(0, T; L^6(\Omega))}
                                                               \|\theta_{\ep}(t)\|_{L^3(\Omega)} 
\\ \notag 
&\leq C_{1}\|\lambda_{\ep}'(\underline{\varphi}_{h})\|_{L^{\infty}(0, T; L^6(\Omega))}
                                                                                   \|\theta_{\ep}(t)\|_{V}.  
\end{align}
Thus, thanks to \eqref{b10}, \eqref{pipi1}, \eqref{pipi2}, 
Lemmas \ref{firstestiPtaueph}, \ref{secondestiPtaueph} and \ref{4estiPtaueph}, 
we can obtain Lemma \ref{5estiPtaueph}.
\end{proof}

\begin{lem}\label{6estiPtaueph}
Let $h_{0}$, $\ep_{1}$ be as in Lemma \ref{firstestiPtaueph}. 
Then there exists a constant $C>0$ depending on the data such that  
$$
\ep\|\overline{u}_{h}\|_{L^{\infty}(0, T; H)}^2 
\leq C 
$$
for all $h \in (0, h_{0})$, $\ep \in (0, \ep_{1})$ and $\tau > 0$.   
\end{lem}
\begin{proof}
This lemma is an immediate consequence of 
\eqref{uj} and Lemma \ref{firstestiPtaueph}. 
\end{proof}

\begin{lem}\label{7estiPtaueph}
Let $h_{0}$, $\ep_{1}$ be as in Lemma \ref{firstestiPtaueph}. 
Then there exists a constant $C>0$ depending on the data such that  
$$
\|c_{s}\partial_{t}\hat{u}_{h}+\lambda_{\ep}'(\underline{\varphi}_{h})
                                                \partial_{t}\hat{\varphi}_{h}\|_{L^2(0, T; V^{*})}^2
\leq C 
$$
for all $h \in (0, h_{0})$, $\ep \in (0, \ep_{1})$ and $\tau > 0$.   
\end{lem}
\begin{proof}
By \eqref{innerVstar} 
taking $v = F^{-1}(c_{s}\partial_{t}\hat{u}_{h}(t) 
                                                 + \lambda_{\ep}'(\underline{\varphi}_{h}(t))
                                                                    \partial_{t}\hat{\varphi}_{h}(t))$ 
in \eqref{dfPtauepnsol1} 
means that 
\begin{align*}
&\|c_{s}\partial_{t}\hat{u}_{h}(t) + \lambda_{\ep}'(\underline{\varphi}_{h}(t))
                                                           \partial_{t}\hat{\varphi}_{h}(t)\|_{V^{*}}^2 
\\ \notag 
&=-\eta\int_{\Omega}\nabla\overline{\theta}_{h}(t)\cdot
                    \nabla F^{-1}(c_{s}\partial_{t}\hat{u}_{h}(t) 
                                                 + \lambda_{\ep}'(\underline{\varphi}_{h}(t))
                                                                        \partial_{t}\hat{\varphi}_{h}(t)) 
\\ \notag 
&\,\quad- \int_{\Gamma}
         \alpha_{\Gamma}(\overline{\theta}_{h}(t)-\overline{\theta_{\Gamma}}_{h}(t))
                              F^{-1}(c_{s}\partial_{t}\hat{u}_{h}(t) 
                                                 + \lambda_{\ep}'(\underline{\varphi}_{h}(t))
                                                                        \partial_{t}\hat{\varphi}_{h}(t)) 
\\ \notag 
&\,\quad+\bigl(
          \overline{f}_{h}(t), F^{-1}(c_{s}\partial_{t}\hat{u}_{h}(t) 
                                                 + \lambda_{\ep}'(\underline{\varphi}_{h}(t))
                                                                        \partial_{t}\hat{\varphi}_{h}(t)) 
             \bigr)_{H}   
\end{align*}
and then using the Young inequality, \eqref{defF}, \eqref{innerVstar} 
and \eqref{keyineq} yields that 
there exists a constant $C_{1} > 0$ satisfying 
\begin{align*}
\|c_{s}\partial_{t}\hat{u}_{h}(t) + \lambda_{\ep}'(\underline{\varphi}_{h}(t))
                                                           \partial_{t}\hat{\varphi}_{h}(t)\|_{V^{*}}^2 
\leq C_{1}\|\overline{\theta}_{h}(t)\|_{V}^2  
       + C_{1}\|\overline{\theta_{\Gamma}}_{h}(t)\|_{L^2(\Gamma)}^2 
       + C_{1}\|\overline{f}_{h}(t)\|_{H}^2.  
\end{align*}
Thus it follows from (C7) and Lemma \ref{firstestiPtaueph} that 
\begin{align*}
\|c_{s}\partial_{t}\hat{u}_{h}+\lambda_{\ep}'(\underline{\varphi}_{h})
                                                \partial_{t}\hat{\varphi}_{h}\|_{L^2(0, T; V^{*})}^2 
\leq C_{2}
\end{align*}
for all $h \in (0, h_{0})$, $\ep \in (0, \ep_{1})$, $\tau > 0$  
and for some constant $C_{2} > 0$. 
\end{proof}
\begin{lem}\label{8estiPtaueph}
Let $h_{0}$, $\ep_{1}$ be as in Lemma \ref{firstestiPtaueph} 
and let $\overline{\tau}$ be as in Theorem \ref{maintheorem1}.   
Then there exists a constant $C>0$ depending on the data such that  
$$
\|\partial_{t}\hat{u}_{h}\|_{L^2(0, T; V^{*})}^2 
+\|\lambda_{\ep}'(\underline{\varphi}_{h})
                            \partial_{t}\hat{\varphi}_{h}\|_{L^2(0, T; L^{3/2}(\Omega))}^2
\leq C(1+\tau^{-1})   
$$
for all $h \in (0, h_{0})$, $\ep \in (0, \ep_{1})$ and $\tau \in (0, \overline{\tau})$.  
\end{lem}
\begin{proof}
Since we have 
\begin{align*}
\|\lambda_{\ep}'(\underline{\varphi}_{h})
                            \partial_{t}\hat{\varphi}_{h}\|_{L^2(0, T; L^{3/2}(\Omega))}^2 
\leq \|\lambda_{\ep}'(\underline{\varphi}_{h})\|_{L^{\infty}(0, T; L^6(\Omega))}^2
                                                 \|\partial_{t}\hat{\varphi}_{h}\|_{L^2(0, T; H)}^2,  
\end{align*}
we can conclude that Lemma \ref{8estiPtaueph} holds 
by \eqref{b10}, Lemmas \ref{firstestiPtaueph} and \ref{7estiPtaueph}.  
\end{proof}

\begin{lem}\label{9estiPtaueph}
Let $h_{0}$, $\ep_{1}$ be as in Lemma \ref{firstestiPtaueph} 
and let $\overline{\tau}$ be as in Theorem \ref{maintheorem1}.   
Then there exists a constant $C>0$ depending on the data such that  
\begin{align*}
&\tau\|\hat{\varphi}_{h}\|_{H^1(0, T; H)}^2 
+ \|\hat{\varphi}_{h}\|_{L^{\infty}(0, T; V)}^2 
+ h\|\hat{\mu}_{h}\|_{L^{\infty}(0, T; H)}^2 
+ h^2\|\hat{\mu}_{h}\|_{H^1(0, T; H)}^2 
\\ 
&+\frac{\ep\tau}{1+\tau}\|\hat{u}_{h}\|_{H^1(0, T; V^{*})}^2   
  +\ep\|\hat{u}_{h}\|_{L^{\infty}(0, T; H)}^2 
\leq C 
\end{align*}
for all $h \in (0, h_{0})$, $\ep \in (0, \ep_{1})$ and $\tau \in (0, \overline{\tau})$.  
\end{lem}
\begin{proof}
This lemma can be obtained by \eqref{rem1}-\eqref{rem6}, 
Lemmas \ref{firstestiPtaueph}, \ref{6estiPtaueph} and \ref{8estiPtaueph}.
\end{proof}

The following lemma asserts strong convergences of 
$\overline{f}_{h}$ and $\overline{\theta_{\Gamma}}_{h}$.  
\begin{lem}\label{conv.f.thetaGamma}
We have that 
$$
\overline{f}_{h} \to f  \quad \mbox{strongly in}\ L^2(0, T; H)
$$
and 
$$
\overline{\theta_{\Gamma}}_{h} \to \theta_{\Gamma} 
\quad \mbox{strongly in}\ L^2(0, T; L^2(\Gamma))
$$
as $h\searrow0$.
\end{lem}
\begin{proof} 
The above convergences hold in general; 
for a proof see, for instance, \cite[Section 5]{CK1}.  
\end{proof}

\begin{prlem3.1}
Let $\ep \in (0, \ep_{1})$ and $\tau \in (0, \overline{\tau})$, 
where $\ep_{1}$ and $\overline{\tau}$ are as in Lemma \ref{firstestiPtaueph}   
and Theorem \ref{maintheorem1}, respectively.    
Then we see from Lemmas \ref{firstestiPtaueph}, \ref{secondestiPtaueph}, 
\ref{5estiPtaueph}, \ref{9estiPtaueph}, the Ascoli--Arzela theorem, 
\eqref{rem7} and \eqref{rem8} 
that there exist some functions 
\begin{align*}
&\varphi \in H^1(0, T; H) \cap L^{\infty}(0, T; V) \cap L^2(0, T; W), \\[1mm]
&\theta \in L^2(0, T; V), \\[1mm] 
&u \in H^1(0, T; V^{*}) \cap L^{\infty}(0, T; H), \\[1mm] 
&w \in L^2(0, T; L^{3/2}(\Omega)), \\[1mm] 
&\mu \in L^2(0, T; V) 
\end{align*} 
such that, possibly for a subsequence $h_{j}$,  
\begin{align}
&\hat{\varphi}_{h} \to \varphi 
\quad \mbox{weakly$^*$ in}\ H^1(0, T; H) 
\cap L^{\infty}(0, T; V), \label{Ptaueph.conv1} 
\\[1.5mm] 
&\overline{\varphi}_{h} \to \varphi 
\quad \mbox{weakly in}\ L^2(0, T; W),  \label{Ptaueph.conv2} 
\\[1.5mm] 
&\hat{\varphi}_{h} \to \varphi 
\quad \mbox{strongly in}\ C([0, T]; H), \label{Ptaueph.conv3} 
\\[1.5mm] 
&\overline{\theta}_{h} \to \theta
\quad \mbox{weakly in}\ L^2(0, T; V) \subset L^2(0, T; L^2(\Gamma)), 
\label{Ptaueph.conv4} 
\\[1.5mm] 
&\hat{u}_{h} \to u
\quad \mbox{weakly$^{*}$ in}\ H^1(0, T; V^{*}) \cap L^{\infty}(0, T; H), 
\label{Ptaueph.conv5}
\\[1.5mm] 
&\hat{u}_{h} \to u
\quad \mbox{strongly in}\ C([0, T]; V^{*}), 
\label{Ptaueph.conv6} 
\\[1.5mm] 
&\lambda_{\ep}'(\underline{\varphi}_{h})\partial_{t}\hat{\varphi}_{h} \to w  
\quad \mbox{weakly$^{*}$ in}\ L^2(0, T; L^{3/2}(\Omega)) 
                                                        \subset L^2(0, T; V^{*}), 
\label{Ptaueph.conv7} 
\\[1.5mm] 
&h\hat{\mu}_{h} \to 0 \quad \mbox{weakly in}\ H^1(0, T; H),  
\label{Ptaueph.conv8} 
\\[1.5mm] 
&\overline{\mu}_{h} \to \mu \quad \mbox{weakly in}\ L^2(0, T; V) 
\label{Ptaueph.conv9}
\end{align}
as $h = h_{j} \searrow 0$. 
Combining \eqref{Ptaueph.conv6} and \eqref{rem8} implies that 
\begin{align}
&\overline{u}_{h} \to u  
\quad \mbox{strongly in}\ L^2(0, T; V^{*})   \label{Ptaueph.conv10} 
\end{align}
as $h = h_{j} \searrow 0$. 
Also, from \eqref{Ptaueph.conv3} and \eqref{rem7} we can observe that 
\begin{align}
&\overline{\varphi}_{h} \to \varphi 
\quad \mbox{strongly in}\ L^2(0, T; H) \label{Ptaueph.conv11} 
\end{align}
as $h = h_{j} \searrow 0$. 
Moreover, it follows from Lemma \ref{firstestiPtaueph}, \eqref{Ptaueph.conv11} 
and \eqref{rem9} that   
\begin{align}\label{Ptaueph.conv12} 
&\underline{\varphi}_{h} \to \varphi 
\quad \mbox{strongly in}\ L^2(0, T; H)  
\end{align}
as $h = h_{j} \searrow 0$.  
Since from \eqref{Ptaueph.conv4} and \eqref{Ptaueph.conv10} 
it turns out that  
\begin{align*}
\int_{0}^{T} \bigl(\mbox{\rm Ln$_{\ep}$}(\overline{\theta}_{h}(t)), 
                                                             \overline{\theta}_{h}(t) \bigr)_{H}\,dt 
&= \int_{0}^{T} \langle 
                          \overline{u}_{h}(t), \overline{\theta}_{h}(t) 
                  \rangle_{V^{*}, V}\,dt 
\\ \notag 
&\to \int_{0}^{T} \langle u(t), \theta(t) \rangle_{V^{*}, V}\,dt 
= \int_{0}^{T} (u(t), \theta(t))_{H}\,dt
\end{align*}
as $h = h_{j} \searrow 0$, the identity 
\begin{align}\label{u.eq.Lnep}
u = \mbox{\rm Ln$_{\ep}$}(\theta)
\end{align} 
holds a.e.\ on $\Omega\times(0, T)$ (see, e.g., \cite[Lemma 1.3, p.\ 42]{Barbu1}). 
Now we let $\psi \in C_{\mathrm c}^{\infty}(\Omega\times(0, T))$. 
Then we derive from the Lipschitz continuity of $\lambda_{\ep}'$, 
\eqref{Ptaueph.conv1} and \eqref{Ptaueph.conv12} that 
\begin{align}\label{for.obtain.w.1}
&\Bigl|
\int_{0}^{T}
   \Bigl(\int_{\Omega}\bigl(\lambda_{\ep}'(\underline{\varphi}_{h} (t))
                                                           \partial_{t}\hat{\varphi}_{h} (t)  
           -\lambda_{\ep}'(\varphi(t))\partial_{t}\varphi(t) \bigr)\psi(t)  
   \Bigr)\,dt  
\Bigr| 
\\ \notag 
&\leq \Bigl|
           \int_{0}^{T}\bigl(\psi(t)(\lambda_{\ep}'(\underline{\varphi}_{h}(t))
                -\lambda_{\ep}'(\varphi(t))), \partial_{t}\hat{\varphi}_{h}(t)\bigr)_{H}\,dt 
\Bigr| 
\\ \notag 
&\,\quad + \Bigl|
           \int_{0}^{T}\bigl(\partial_{t}\hat{\varphi}_{h}(t)-\partial_{t}\varphi(t), 
                                                   \lambda_{\ep}'(\varphi(t))\psi(t) \bigr)_{H}\,dt 
\Bigr| 
\\ \notag 
& \to 0
\end{align}
as $h = h_{j} \searrow 0$. 
On the other hand, the convergence \eqref{Ptaueph.conv7} yields that 
\begin{align}\label{for.obtain.w.2}
\int_{0}^{T}
   \Bigl(\int_{\Omega}\lambda_{\ep}'(\underline{\varphi}_{h} (t))
                                                      \partial_{t}\hat{\varphi}_{h}(t)\psi(t)  
   \Bigr)\,dt  
&= \int_{0}^{T}
    \bigl\langle
      \lambda_{\ep}'(\underline{\varphi}_{h} (t))\partial_{t}\hat{\varphi}_{h}(t), \psi(t)
     \bigr\rangle_{L^{3/2}(\Omega), L^3(\Omega)}\,dt  
\\ \notag 
&\to \int_{0}^{T}
      \langle
      w(t), \psi(t)
      \rangle_{L^{3/2}(\Omega), L^3(\Omega)}\,dt  
\\ \notag 
&= \int_{0}^{T}\Bigl(\int_{\Omega}w(t)\psi(t) \Bigr)\,dt  
\end{align}
as $h = h_{j} \searrow 0$. 
Hence, thanks to \eqref{for.obtain.w.1} and \eqref{for.obtain.w.2},  
we can verify that  
\begin{align*}
\int_{\Omega\times(0, T)}
\bigl(w-\lambda_{\ep}'(\varphi)\partial_{t}\varphi\bigr)\psi 
= 0 
\end{align*}
for all $\psi \in C_{\mathrm c}^{\infty}(\Omega\times(0, T))$, 
which means that 
\begin{align}\label{obtain.w}
w = \lambda_{\ep}'(\varphi)\partial_{t}\varphi = \partial_{t}\lambda_{\ep}(\varphi)
\end{align}
a.e.\ on $\Omega\times(0, T)$. 
Thus in view of \ref{Ptaueph} 
we can obtain \eqref{dfPtauepsol1} from 
\eqref{Ptaueph.conv5}, \eqref{u.eq.Lnep}, 
\eqref{Ptaueph.conv7}, \eqref{obtain.w}, \eqref{Ptaueph.conv4} 
and Lemma \ref{conv.f.thetaGamma}. 
In addition, \eqref{dfPtauepsol2} is a consequence of 
\eqref{Ptaueph.conv1}, \eqref{Ptaueph.conv8} and \eqref{Ptaueph.conv9}, 
while the initial conditions \eqref{dfPtauepsol4} 
follow from \eqref{Ptaueph.conv6}, \eqref{u.eq.Lnep} and \eqref{Ptaueph.conv3}. 

Next we show that  
\begin{align}\label{Ptauephenbun}
\int_{\Omega\times(0, T)} 
                  \bigl(\mu - \tau\partial_{t}\varphi 
                          + \gamma\Delta\varphi 
                          - \beta_{\ep}(\varphi) - \sigma'(\varphi) 
                         + \lambda_{\ep}'(\varphi)\theta \bigr)\psi 
= 0
\end{align}
for all $\psi \in C_{\mathrm c}^{\infty}(\Omega\times(0, T))$. 
We have from \eqref{Ptaueph.conv1}, \eqref{Ptaueph.conv2}, 
\eqref{Ptaueph.conv4}, \eqref{Ptaueph.conv9}, \eqref{Ptaueph.conv11}, 
\eqref{Ptaueph.conv12}, 
the Lipschitz continuity of $\beta_{\ep}$ and $\lambda_{\ep}'$ 
that 
\begin{align*}
0&=\int_{0}^{T} \Bigl(\int_{\Omega} 
                   \bigl(\overline{\mu}_{h}(t) - \tau\partial_{t}\hat{\varphi}_{h}(t) 
                     + \gamma\Delta\overline{\varphi}_{h}(t) 
                     - \beta_{\ep}(\overline{\varphi}_{h}(t)) 
                     - \sigma'(\overline{\varphi}_{h}(t)) 
\\ \notag 
&\hspace{9cm} + \lambda_{\ep}'(\underline{\varphi}_{h}(t))\overline{\theta}_{h}(t) 
                                                                                                \bigr)\psi(t) 
               \Bigr)\,dt 
\\ \notag 
&=\int_{0}^{T} \bigl(\overline{\mu}_{h}(t) - \tau\partial_{t}\hat{\varphi}_{h}(t) 
                     + \gamma\Delta\overline{\varphi}_{h}(t) 
                     - \beta_{\ep}(\overline{\varphi}_{h}(t)) 
                     - \sigma'(\overline{\varphi}_{h}(t)), \psi(t) \bigr)_{H}\,dt 
\\ \notag 
&\,\quad + \int_{0}^{T} \bigl(\psi(t)\lambda_{\ep}'(\underline{\varphi}_{h}(t)), 
                                                              \overline{\theta}_{h}(t)  \bigr)_{H}\,dt 
\\[5mm] \notag 
&\to \int_{0}^{T} \bigl(\mu(t) - \tau\partial_{t}\varphi(t) 
                          + \gamma\Delta\varphi(t) 
                          - \beta_{\ep}(\varphi(t)) 
                          - \sigma'(\varphi(t)), \psi(t) \bigr)_{H}\,dt 
\\ \notag 
&\,\quad + \int_{0}^{T} \bigl(\psi(t)\lambda_{\ep}'(\varphi(t)), 
                                                                  \theta(t) \bigr)_{H}\,dt 
\end{align*}
as $h = h_{j} \searrow 0$. 
Thus \eqref{Ptauephenbun} holds. 
Then we can conclude that \eqref{dfPtauepsol3} holds.    

Therefore Lemma \ref{existPtauep} is completely proved. 
\qed
\end{prlem3.1}

\vspace{10pt}

%%==============================================================%%
%%==============                                  ==============%%
%%======                      Section5                    ======%%
%%====                                                      ====%%
%%==                                                          ==%%
%%====                                                      ====%%
%%======                                                  ======%%
%%==============                                  ==============%%
%%==============================================================%%

\section{Estimates for \ref{Ptauep} and passage to the limit as $\ep\searrow0$} 
\label{Sec5}

In this section we will confirm that Theorem \ref{maintheorem1} holds. 
We will establish estimates for \ref{Ptauep} 
in order to show existence for \ref{Ptau}  
by passing to the limit in \ref{Ptauep} as $\ep\searrow0$.  
\begin{lem}\label{1estiPtauep} 
Let $\ep_{1}$ be as in Lemma \ref{firstestiPtaueph} 
and let $\overline{\tau}$ be as in Theorem \ref{maintheorem1}.  
Then there exists a constant $C>0$ depending on the data such that  
\begin{align*}
&\tau\|\partial_{t}\varphi_{\ep}\|_{L^2(0, T; H)}^2 
+ \|\varphi_{\ep}\|_{L^{\infty}(0, T; V)}^2 
+ \|\theta_{\ep}\|_{L^2(0, T; V)}^2 
+ \ep\|\theta_{\ep}\|_{L^{\infty}(0, T; H)}^2 
+ \|\partial_{t}\varphi_{\ep}\|_{L^2(0, T; V^{*})}^2 
\\ \notag 
&+ \|\mu_{\ep}\|_{L^2(0, T; V)}^2 
+ \|\beta_{\ep}(\varphi_{\ep})\|_{L^2(0, T; H)}^2 
+ \|\varphi_{\ep}\|_{L^2(0, T; W)}^2 
+ \|(c_{s}\mbox{\rm Ln$_{\ep}$}(\theta_{\ep})
                           +\lambda_{\ep}(\varphi_{\ep}))_{t}\|_{L^2(0, T; V^{*})}^2 
\\ \notag 
&+ \frac{\tau}{1+\tau}\|\partial_{t}\lambda_{\ep}(\varphi_{\ep})\|_{L^2(0, T; V^{*})}^2
+ \frac{\tau}{1+\tau}\|(\mbox{\rm Ln$_{\ep}$}(\theta_{\ep}))_{t}\|_{L^2(0, T; V^{*})}^2
\\ \notag  
&\leq C
\end{align*}
for all $\ep\in(0, \ep_{1})$ and all $\tau \in (0, \overline{\tau})$.    
\end{lem}
\begin{proof}
Combining Lemmas \ref{firstestiPtaueph}-\ref{8estiPtaueph} 
leads to Lemma \ref{1estiPtauep}. 
\end{proof}

\begin{lem}\label{2estiPtauep}
Let $\ep_{1}$ be as in Lemma \ref{firstestiPtaueph} 
and let $\overline{\tau}$ be as in Theorem \ref{maintheorem1}.  
Then there exists a constant $C>0$ depending on the data such that  
\begin{align*} 
\|\lambda_{\ep}(\varphi_{\ep})\|_{L^{\infty}(0, T; H)}^2
+ \|\mbox{\rm Ln$_{\ep}$}(\theta_{\ep})\|_{L^{\infty}(0, T; H)}^2
\leq C
\end{align*}
for all $\ep\in(0, \ep_{1})$ and all $\tau \in (0, \overline{\tau})$.    
\end{lem}
\begin{proof}
The Taylor formula with integral remainder and \eqref{lamep2} imply that  
\begin{align*}
|\lambda_{\ep}(\varphi_{\ep}(t))| 
&\leq 
|\lambda_{\ep}(0)| + |\lambda_{\ep}'(0)||\varphi_{\ep}(t)| 
       + \frac{\|\lambda_{\ep}''\|_{L^{\infty}(\mathbb{R})}}{2}|\varphi_{\ep}(t)|^2   
\\ \notag 
&\leq M_{\lambda}\left(1+|\varphi_{\ep}(t)|+\frac{1}{2}|\varphi_{\ep}(t)|^2 \right). 
\end{align*}
Then we see from 
the continuity of the embedding $V \hookrightarrow L^4(\Omega)$ 
and Lemma \ref{1estiPtauep}  
that 
\begin{align}\label{can1}
\|\lambda_{\ep}(\varphi_{\ep}(t))\|_{H}^2 
\leq C_{1} + C_{1}\|\varphi_{\ep}(t)\|_{L^4(\Omega)}^4 
\leq C_{1} + C_{2}\|\varphi_{\ep}(t)\|_{V}^4 
\leq C_{3} 
\end{align} 
for all $\ep \in (0, \ep_{1})$, $\tau \in (0, \overline{\tau})$,  
a.a.\ $t \in (0, T)$ and for some constants $C_{1}, C_{2}, C_{3} > 0$. 
Next we take $v = c_{s}\mbox{\rm Ln$_{\ep}$}(\theta_{\ep}(t))
                                            +\lambda_{\ep}(\varphi_{\ep}(t))$ 
in \eqref{dfPtauepsol1} 
to derive that 
\begin{align}\label{can2}
&\frac{1}{2}\frac{d}{dt}\|c_{s}\mbox{\rm Ln$_{\ep}$}(\theta_{\ep}(t))
                                            +\lambda_{\ep}(\varphi_{\ep}(t))\|_{H}^2 
+ c_{s}\eta\int_{\Omega}\nabla\theta_{\ep}(t)
                                   \cdot\nabla\mbox{\rm Ln$_{\ep}$}(\theta_{\ep}(t)) 
\\ \notag 
&+c_{s}\int_{\Gamma} \alpha_{\Gamma}(\theta_{\ep}(t)-\theta_{\Gamma}(t))
                        (\mbox{\rm Ln$_{\ep}$}(\theta_{\ep}(t))
                                                 -\mbox{\rm Ln$_{\ep}$}(\theta_{\Gamma}(t))) 
\\ \notag 
&=\bigl(f(t), c_{s}\mbox{\rm Ln$_{\ep}$}(\theta_{\ep}(t))
                                            +\lambda_{\ep}(\varphi_{\ep}(t)) \bigr)_{H} 
\\ \notag 
 &\,\quad- \int_{\Gamma}\alpha_{\Gamma}(\theta_{\ep}(t)-\theta_{\Gamma}(t))
                                         (c_{s}\mbox{\rm Ln$_{\ep}$}(\theta_{\Gamma}(t)) 
                                                                    + \lambda_{\ep}(\varphi_{\ep}(t))) 
- \eta \int_{\Omega}\nabla\theta_{\ep}(t)\cdot
                                       \nabla\lambda_{\ep}(\varphi_{\ep}(t)). 
\end{align}
Here it follows from Remark \ref{def.of.rhoep} that 
\begin{align}\label{can3}
&c_{s}\eta\int_{\Omega}\nabla\theta_{\ep}(t)
                                     \cdot\nabla\mbox{\rm Ln$_{\ep}$}(\theta_{\ep}(t)) 
\\ \notag 
&=  c_{s}\eta\ep\|\nabla\theta_{\ep}(t)\|_{H}^2  
     + c_{s}\eta\ep\|\nabla\ln_{\ep}(\theta_{\ep}(t))\|_{H}^2 
     + c_{s}\eta\int_{\Omega}\frac{|\nabla\rho_{\ep}(\theta_{\ep}(t))|^2}
                                                             {\rho_{\ep}(\theta_{\ep}(t))}. 
\end{align}
The monotonicity of $\mbox{\rm Ln$_{\ep}$}$ and (C4) entail that 
\begin{align}\label{can4}
c_{s}\int_{\Gamma} \alpha_{\Gamma}(\theta_{\ep}(t)-\theta_{\Gamma}(t))
                        (\mbox{\rm Ln$_{\ep}$}(\theta_{\ep}(t))
                                                 -\mbox{\rm Ln$_{\ep}$}(\theta_{\Gamma}(t))) 
\geq 0. 
\end{align}
From (C7) we can observe that   
\begin{align}\label{can5}
|\mbox{\rm Ln$_{\ep}$}(\theta_{\Gamma}(t))| 
\leq \ep|\theta_{\Gamma}| + |\ln \theta_{\Gamma}(t)| 
\leq \theta^{*} + \max_{\theta_{*} \leq r \leq \theta^{*}}|\ln r|
\end{align}
for all $\ep \in (0, 1]$ and a.a.\ $t \in (0, T)$.  
We deduce from \eqref{keyineq}, \eqref{lamep2}, 
the continuity of the embedding $V \hookrightarrow L^4(\Omega)$ 
and Lemma \ref{1estiPtauep} that  
\begin{align}\label{can6}
&\|\lambda_{\ep}(\varphi_{\ep}(t))\|_{L^2(\Gamma)}^2 
  + \|\nabla\lambda_{\ep}(\varphi_{\ep}(t))\|_{H}^2 
\\ \notag 
&\leq \frac{1}{C_{*}}\|\lambda_{\ep}(\varphi_{\ep}(t))\|_{V}^2  
\\ \notag 
&= \frac{1}{C_{*}}\|\lambda_{\ep}(\varphi_{\ep}(t))\|_{H}^2 
     + \frac{1}{C_{*}}\int_{\Omega}
                            |\lambda_{\ep}'(\varphi_{\ep}(t))|^2 |\nabla \varphi_{\ep}(t)|^2 
\\ \notag
&\leq \frac{1}{C_{*}}\|\lambda_{\ep}(\varphi_{\ep}(t))\|_{H}^2 
     + \frac{2\|\lambda_{\ep}''\|_{L^{\infty}(\mathbb{R})}^2}{C_{*}}
                                             \|\varphi_{\ep}(t)\|_{L^4(\Omega)}^2
                                                      \|\nabla \varphi_{\ep}(t)\|_{L^4(\Omega)}^2  
     + \frac{2|\lambda_{\ep}'(0)|^2}{C_{*}}\|\nabla \varphi_{\ep}(t)\|_{H}^2 
\\ \notag 
&\leq \frac{C_{3}}{C_{*}} 
         + C_{4}(1 + \|\varphi_{\ep}\|_{L^{\infty}(0, T; V)}^2)\|\varphi_{\ep}(t)\|_{W}^2 
\\ \notag 
&\leq \frac{C_{3}}{C_{*}} + C_{5}\|\varphi_{\ep}(t)\|_{W}^2 
\end{align}
for all $\ep \in (0, \ep_{1})$, $\tau \in (0, \overline{\tau})$,  
a.a.\ $t \in (0, T)$ and for some constants $C_{3}, C_{4}, C_{5} > 0$.  
Thus we have from \eqref{can2}-\eqref{can6} and the Young inequality 
that there exists a constant $C_{6} > 0$ satisfying 
\begin{align}\label{can7}
&\frac{1}{2}\frac{d}{dt}\|c_{s}\mbox{\rm Ln$_{\ep}$}(\theta_{\ep}(t))
                                            +\lambda_{\ep}(\varphi_{\ep}(t))\|_{H}^2 
\\ \notag 
&\leq \frac{1}{2}\|f(t)\|_{H}^2 
         + \frac{1}{2}\|c_{s}\mbox{\rm Ln$_{\ep}$}(\theta_{\ep}(t))
                                            +\lambda_{\ep}(\varphi_{\ep}(t))\|_{H}^2 
\\ \notag 
    &\,\quad+ C_{6}\|\theta_{\ep}(t)\|_{V}^2 
                + C_{6}\|\theta_{\ep}(t)\|_{L^2(\Gamma)}^2 
                + C_{6}\|\varphi_{\ep}(t)\|_{W}^2 + C_{6}  
\end{align}
for all $\ep \in (0, \ep_{1})$, $\tau \in (0, \overline{\tau})$ and a.a.\ $t \in (0, T)$. 
Here, 
since we can obtain that 
\begin{align*}
|c_{s}\mbox{\rm Ln$_{\ep}$}(\theta_{0})+\lambda_{\ep}(\varphi_{0})| 
&\leq c_{s}\ep|\theta_{0}|  + c_{s}|\ln \theta_{0}| 
       + |\lambda_{\ep}(0)| + |\lambda_{\ep}'(0)||\varphi_{0}| 
       + \frac{\|\lambda_{\ep}''\|_{L^{\infty}(\mathbb{R})}}{2}|\varphi_{0}|^2  
\\ \notag 
&\leq c_{s}|\theta_{0}| + c_{s}\max_{\theta_{*} \leq r \leq \theta^{*}}|\ln r| 
             + M_{\lambda}\left(1+|\varphi_{0}|+\frac{1}{2}|\varphi_{0}|^2 \right)
\end{align*}
by the Taylor formula with integral remainder, (C7) and \eqref{lamep2}, 
the identities 
\begin{align}\label{can8}
\|c_{s}\mbox{\rm Ln$_{\ep}$}(\theta_{0})+\lambda_{\ep}(\varphi_{0})\|_{H}^2 
\leq C_{7} + C_{7}\|\varphi_{0}\|_{L^4(\Omega)}^4 
\leq C_{7} + C_{8}\|\varphi_{0}\|_{V}^4 
\end{align}
hold for all $\ep \in (0, \ep_{1})$ and for some constants $C_{7}, C_{8} > 0$. 
Hence, integrating \eqref{can7} over $(0, t)$, where $t \in [0, T]$, 
and using \eqref{can8}, (C7), and Lemma \ref{1estiPtauep}, we conclude that 
there exists a constant $C_{9} > 0$ such that 
\begin{align*}
\frac{1}{2}\|c_{s}\mbox{\rm Ln$_{\ep}$}(\theta_{\ep}(t))
                                            +\lambda_{\ep}(\varphi_{\ep}(t))\|_{H}^2 
\leq  \frac{1}{2}\int_{0}^{t}\|c_{s}\mbox{\rm Ln$_{\ep}$}(\theta_{\ep}(s))
                                            +\lambda_{\ep}(\varphi_{\ep}(s))\|_{H}^2\,ds  
        + C_{9}
\end{align*}
for all $\ep \in (0, \ep_{1})$, $\tau \in (0, \overline{\tau})$ and a.a.\ $t \in (0, T)$. 
Therefore by applying the Gronwall lemma 
there exists a constant $C_{10} > 0$ satisfying 
\begin{align*}
\|c_{s}\mbox{\rm Ln$_{\ep}$}(\theta_{\ep}(t))+\lambda_{\ep}(\varphi_{\ep}(t))\|_{H}^2 
\leq  C_{10}
\end{align*} 
for all $\ep \in (0, \ep_{1})$, $\tau \in (0, \overline{\tau})$ and a.a.\ $t \in (0, T)$, 
which leads to Lemma \ref{2estiPtauep} by \eqref{can1}. 
\end{proof}

\begin{prth2.1}
Let $\tau \in (0, \overline{\tau})$, 
where $\overline{\tau}$ is as in Theorem \ref{maintheorem1}.  
Then we combine Lemmas \ref{1estiPtauep}, \ref{2estiPtauep}, 
the Aubin--Lions lemma and the Ascoli--Arzela theorem 
to infer that there exist some functions 
\begin{align*}
&\varphi \in H^1(0, T; H) \cap L^{\infty}(0, T; V) \cap L^2(0, T; W), \\[1mm]
&\theta \in L^2(0, T; V), \\[1mm] 
&u \in H^1(0, T; V^{*}) \cap L^{\infty}(0, T; H), \\[1mm]  
&w \in H^1(0, T; V^{*}) \cap L^{\infty}(0, T; H), \\[1mm]  
&\mu \in L^2(0, T; V), \\[1mm]  
&\xi \in L^2(0, T; H)  
\end{align*} 
such that, possibly for a subsequence $\ep_{j}$,  
\begin{align}
&\varphi_{\ep} \to \varphi 
\quad \mbox{weakly$^*$ in}\ H^1(0, T; H) 
\cap L^{\infty}(0, T; V) \cap L^2(0, T; W), \label{Ptauep.conv1} 
\\[1.5mm] 
&\varphi_{\ep} \to \varphi 
\quad \mbox{strongly in}\ C([0, T]; H) \cap L^2(0, T; V), \label{Ptauep.conv2} 
\\[1.5mm] 
&\theta_{\ep} \to \theta
\quad \mbox{weakly in}\ L^2(0, T; V) \subset L^2(0, T; L^2(\Gamma)), 
\label{Ptauep.conv3} 
\\[1.5mm] 
&\ep\theta_{\ep} \to 0  
\quad \mbox{strongly in}\ L^{\infty}(0, T; H), 
\label{Ptauep.conv4} 
\\[1.5mm] 
&\mbox{\rm Ln$_{\ep}$}(\theta_{\ep}) \to u
\quad \mbox{weakly$^{*}$ in}\ H^1(0, T; V^{*}) \cap L^{\infty}(0, T; H),  
\label{Ptauep.conv5} 
\\[1.5mm] 
&\mbox{\rm Ln$_{\ep}$}(\theta_{\ep}) \to u
\quad \mbox{strongly in}\ C([0, T]; V^{*}), 
\label{Ptauep.conv6} 
\\[1.5mm] 
&\lambda_{\ep}(\varphi_{\ep}) \to w  
\quad \mbox{weakly$^{*}$ in}\ H^1(0, T; V^{*}) \cap L^{\infty}(0, T; H), 
\label{Ptauep.conv7} 
\\[1.5mm] 
&\mu_{\ep} \to \mu \quad \mbox{weakly in}\ L^2(0, T; V), 
\label{Ptauep.conv8} 
\\[1.5mm] 
&\beta_{\ep}(\varphi_{\ep}) \to \xi 
\quad \mbox{weakly in}\ L^2(0, T; H) \label{Ptauep.conv9}
\end{align}
as $\ep = \ep_{j} \searrow 0$. 
We have from 
\eqref{Ptauep.conv3}-\eqref{Ptauep.conv6} that   
\begin{align*}
&\ln_{\ep}(\theta_{\ep}) 
= \mbox{\rm Ln$_{\ep}$}(\theta_{\ep}) - \ep\theta_{\ep} 
\to u \quad \mbox{weakly in}\ L^2(0, T; H), 
\\[1mm] 
&\theta_{\ep} \to \theta \quad \mbox{weakly in}\ L^2(0, T; H) 
\end{align*}
as $\ep = \ep_{j} \searrow 0$ and 
\begin{align*}
\int_{0}^{T} \bigl(\ln_{\ep}(\theta_{\ep}(t)), \theta_{\ep}(t) \bigr)_{H}\,dt  
&= \int_{0}^{T} \langle 
                          \mbox{\rm Ln$_{\ep}$}(\theta_{\ep}(t)), \theta_{\ep}(t) 
                  \rangle_{V^{*}, V}\,dt 
     - \ep\int_{0}^{T}\|\theta_{\ep}(t)\|_{H}^2\,dt 
\\ \notag 
&\to \int_{0}^{T} \langle u(t), \theta(t) \rangle_{V^{*}, V}\,dt 
= \int_{0}^{T} (u(t), \theta(t))_{H}\,dt
\end{align*}
as $\ep = \ep_{j} \searrow 0$.   
Hence the identity 
\begin{align}\label{Ptau.ln}
u = \ln \theta   
\end{align}
holds a.e.\ on $\Omega\times(0, T)$ (see, e.g., \cite[Lemma 1.3, p.\ 42]{Barbu1}). 
We see from the Taylor formula with integral remainder and \eqref{lamep2} 
that 
\begin{align*}
&|\lambda_{\ep}(\varphi_{\ep}) - \lambda(\varphi)| 
\\ \notag 
&\leq |\lambda_{\ep}(\varphi_{\ep}) - \lambda_{\ep}(\varphi)| 
      + |\lambda_{\ep}(\varphi) - \lambda(\varphi)|       
\\ \notag 
&\leq |\lambda_{\ep}'(\varphi)| |\varphi_{\ep}-\varphi| 
         + \frac{\|\lambda_{\ep}''\|_{L^{\infty}(\mathbb{R})}}{2}
                                                                        |\varphi_{\ep}-\varphi|^2 
         + |\lambda_{\ep}(\varphi) - \lambda(\varphi)|       
\\ \notag 
&\leq M_{\lambda}|\varphi||\varphi_{\ep}-\varphi| 
         + M_{\lambda}|\varphi_{\ep}-\varphi| 
         + \frac{M_{\lambda}}{2}|\varphi_{\ep}-\varphi|^2 
         + |\lambda_{\ep}(\varphi) - \lambda(\varphi)|,        
\end{align*}
whence we obtain that 
\begin{align*}
&\|\lambda_{\ep}(\varphi_{\ep}) - \lambda(\varphi)\|_{L^2(0, T; H)} 
\\ \notag 
&\leq C_{1}\|\varphi\|_{L^{\infty}(0, T; L^4(\Omega))}
                          \|\varphi_{\ep}-\varphi\|_{L^2(0, T; L^4(\Omega))} 
         + C_{1}\|\varphi_{\ep}-\varphi\|_{L^2(0, T; H)} 
\\ \notag 
&\,\quad+ C_{1}\|\varphi_{\ep}-\varphi\|_{L^{\infty}(0, T; L^4(\Omega))}
                                      \|\varphi_{\ep}-\varphi\|_{L^2(0, T; L^4(\Omega))} 
             + C_{1}\|\lambda_{\ep}(\varphi)-\lambda(\varphi)\|_{L^2(0, T; H)}      
\\ \notag 
&\leq C_{2}\|\varphi\|_{L^{\infty}(0, T; V)}
                          \|\varphi_{\ep}-\varphi\|_{L^2(0, T; V)}  
         + C_{1}\|\varphi_{\ep}-\varphi\|_{L^2(0, T; H)}  
\\ \notag 
 &\,\quad+ C_{2}\|\varphi_{\ep}-\varphi\|_{L^{\infty}(0, T; V)}
                                               \|\varphi_{\ep}-\varphi\|_{L^2(0, T; V)} 
         + C_{1}\|\lambda_{\ep}(\varphi) -\lambda(\varphi)\|_{L^2(0, T; H)}  
\end{align*}
for all $\ep \in (0, \ep_{1})$ and for some constants $C_{1}, C_{2} > 0$. 
Thus, thanks to \eqref{Ptauep.conv1}, \eqref{Ptauep.conv2}, 
\eqref{lamep2}, \eqref{lamep3} and the Lebesgue dominated convergence theorem,  
we can verify that 
\begin{align}\label{strongconv.lamep}
\lambda_{\ep}(\varphi_{\ep}) \to \lambda(\varphi) 
\quad \mbox{strongly in}\ L^2(0, T; H) 
\end{align}
as $\ep = \ep_{j} \searrow 0$. 
Therefore it holds that 
\begin{align}\label{Ptau.w.lambda}
w = \lambda(\varphi)
\end{align}
a.e.\ on $\Omega\times(0, T)$ 
by \eqref{Ptauep.conv7} and \eqref{strongconv.lamep}. 
Hence, in the light of \eqref{Ptau.ln} and \eqref{Ptau.w.lambda}, 
we can prove that $\theta$, $\mu$, $\varphi$ satisfy \eqref{dfPtausol1} 
and \eqref{dfPtausol2} by passing to the limit 
in \eqref{dfPtauepsol1} and \eqref{dfPtauepsol2} 
as $\ep=\ep_{j}$ tends to $0$. 
The initial conditions \eqref{dfPtausol4} follow from \eqref{dfPtauepsol4}, 
due to \eqref{Ptauep.conv6}, \eqref{Ptau.ln} and \eqref{Ptauep.conv2}. 

Next we confirm that  
\begin{align}\label{Ptauhenbun}
&\int_{\Omega\times(0, T)}  \bigl(\mu - \tau\partial_{t}\varphi 
                          + \gamma\Delta\varphi 
                         - \xi - \sigma'(\varphi) 
 + \lambda'(\varphi)\theta \bigr)\psi 
= 0
\end{align}
for all $\psi \in C_{\mathrm c}^{\infty}(\Omega\times(0, T))$. 
Indeed, it follows from \eqref{lamep2} that 
\begin{align*}
&\|\lambda_{\ep}'(\varphi_{\ep}) - \lambda'(\varphi)\|_{L^2(0, T; H)} 
\\ \notag 
&\leq \|\lambda_{\ep}'(\varphi_{\ep}) - \lambda_{\ep}'(\varphi)\|_{L^2(0, T; H)}  
      + \|\lambda_{\ep}'(\varphi) - \lambda'(\varphi)\|_{L^2(0, T; H)}     
\\ \notag 
&\leq  M_{\lambda}\|\varphi_{\ep}-\varphi\|_{L^2(0, T; H)}  
         + \|\lambda_{\ep}'(\varphi) - \lambda'(\varphi)\|_{L^2(0, T; H)}.          
\end{align*}
Thus we deduce from \eqref{Ptauep.conv2}, 
\eqref{lamep2}, \eqref{lamep3} and the Lebesgue dominated convergence theorem 
that  
\begin{align}\label{Ptau.conv.derivative.lambda}
\lambda_{\ep}'(\varphi_{\ep}) \to \lambda'(\varphi) 
\quad \mbox{strongly in}\ L^2(0, T; H) 
\end{align}
as $\ep = \ep_{j} \searrow 0$. 
Then it follows from 
\eqref{Ptauep.conv1}, \eqref{Ptauep.conv2}, \eqref{Ptauep.conv8},  
\eqref{Ptauep.conv9} and \eqref{Ptau.conv.derivative.lambda} 
that    
\begin{align*}
0&=\int_{0}^{T} \Bigl(\int_{\Omega} 
                   \bigl(\mu_{\ep}(t) - \tau\partial_{t}\varphi_{\ep}(t) 
                     + \gamma\Delta\varphi_{\ep}(t) 
                     - \beta_{\ep}(\varphi_{\ep}(t)) - \sigma'(\varphi_{\ep}(t)) 
\\ \notag 
&\hspace{9cm} + \lambda_{\ep}'(\varphi_{\ep}(t))\theta_{\ep}(t) 
                                                                                            \bigr)\psi(t) 
               \Bigr)\,dt 
\\[1mm] \notag 
&=\int_{0}^{T} \bigl(\mu_{\ep}(t) - \tau\partial_{t}\varphi_{\ep}(t) 
                          + \gamma\Delta\varphi_{\ep}(t) 
                          - \beta_{\ep}(\varphi_{\ep}(t)) 
                          - \sigma'(\varphi_{\ep}(t)), \psi(t) \bigr)_{H}\,dt 
\\ \notag 
&\,\quad + \int_{0}^{T} \bigl(\psi(t)\lambda_{\ep}'(\varphi_{\ep}(t)), 
                                                                  \theta_{\ep}(t) \bigr)_{H}\,dt 
\\[7mm] \notag 
&\to \int_{0}^{T} \bigl(\mu(t) - \tau\partial_{t}\varphi(t) 
                          + \gamma\Delta\varphi(t) 
                          - \xi(t) 
                          - \sigma'(\varphi(t)), \psi(t) \bigr)_{H}\,dt 
\\ \notag 
&\,\quad + \int_{0}^{T} \bigl(\psi(t)\lambda'(\varphi(t)), 
                                                                  \theta(t) \bigr)_{H}\,dt 
\end{align*}
as $\ep = \ep_{j} \searrow 0$. 
Hence \eqref{Ptauhenbun} holds, which means that    
\begin{align}\label{semisecondeqPtau}
\mu = \tau\partial_{t}\varphi - \gamma\Delta\varphi 
                    + \xi 
                    + \sigma'(\varphi) - \lambda'(\varphi)\theta   
         \quad \mbox{a.e.\ on}\ \Omega\times(0, T). 
\end{align}
On the other hand, 
the convergences \eqref{Ptauep.conv2} and \eqref{Ptauep.conv9} yield that  
\begin{align*}
\int_{0}^{T} \bigl(\beta(\varphi_{\ep}(t)), \varphi_{\ep}(t) \bigr)_{H}\,dt 
\to \int_{0}^{T} (\xi(t), \varphi(t))_{H}\,dt
\end{align*}
as $\ep = \ep_{j} \searrow 0$   
and then the inclusion
\begin{align}\label{Ptau.beta}
\xi \in \beta(\varphi)  
\end{align} 
holds a.e.\ on $\Omega\times(0, T)$ (see, e.g., \cite[Lemma 1.3, p.\ 42]{Barbu1}). 
Thus combining \eqref{semisecondeqPtau} and \eqref{Ptau.beta} 
leads to \eqref{dfPtausol3}.  

Therefore we can conclude that Theorem \ref{maintheorem1} holds. 
\qed
\end{prth2.1}

\vspace{10pt}

%%==============================================================%%
%%==============                                  ==============%%
%%======                      Section6                    ======%%
%%====                                                      ====%%
%%==                                                          ==%%
%%====                                                      ====%%
%%======                                                  ======%%
%%==============                                  ==============%%
%%==============================================================%%

\section{Estimates for \ref{Ptau} and passage to the limit as $\tau\searrow0$} 
\label{Sec6}

In this section we will prove Theorem \ref{maintheorem2}. 
We will derive existence for \ref{P}  
by passing to the limit in \ref{Ptau} as $\tau\searrow0$.    
%\begin{lem}\label{estiforPtau}
%{\red Let $\overline{\tau}$ be as in Lemma \ref{secondestiPtaueph}.} 
%Then there exists a constant $C>0$ such that  
%\begin{align*} 
%&\tau\|\partial_{t}\varphi_{\tau}\|_{L^2(0, T; H)}^2 
%+ \|\varphi_{\tau}\|_{L^{\infty}(0, T; V)}^2 
%+ \|\theta_{\tau}\|_{L^2(0, T; V)}^2 
%+ \|\partial_{t}\varphi_{\tau}\|_{L^2(0, T; V^{*})}^2 
%\\ 
%&+ \|\mu_{\tau}\|_{L^2(0, T; V)}^2 
%  + \|\xi_{\tau}\|_{L^2(0, T; H)}^2 + \|\varphi_{\tau}\|_{L^2(0, T; W)}^2 
%  + \|(c_{s}\ln \theta_{\tau}+\lambda(\varphi_{\tau}))_{t}\|_{L^2(0, T; V^{*})}^2 
%\\ 
%&+ \|\lambda(\varphi_{\tau})\|_{L^{\infty}(0, T; H)}^2
%  + \|\ln \theta_{\tau}\|_{L^{\infty}(0, T; H)}^2
%\leq C 
%\end{align*}
%for all {\red $\tau \in (0, \overline{\tau})$.}   
%\end{lem}
%\begin{proof}
%We can obtain this lemma by Lemmas \ref{1estiPtauep} and \ref{2estiPtauep}. 
%\end{proof}

\begin{prth2.2}
By Theorem \ref{maintheorem1}    
there exist some functions 
\begin{align*}
&\varphi \in H^1(0, T; V^{*}) \cap L^{\infty}(0, T; V) \cap L^2(0, T; W), \\[1mm] 
&\theta \in L^2(0, T; V), \\[1mm] 
&z \in H^1(0, T; V^{*}) \cap L^{\infty}(0, T; H), \\[1mm] 
&\mu \in L^2(0, T; V), \\[1mm] 
&\xi \in L^2(0, T; H)  
\end{align*}  
such that, possibly for a subsequence $\tau_{j}$,  
\begin{align}
&\varphi_{\tau} \to \varphi  
\quad \mbox{weakly$^*$ in}\ H^1(0, T; V^{*}) 
\cap L^{\infty}(0, T; V) \cap L^2(0, T; W), \label{Ptau.conv1} 
\\[1.5mm] 
&\varphi_{\tau} \to \varphi  
\quad \mbox{strongly in}\ C([0, T]; H) \cap L^2(0, T; V), \label{Ptau.conv2} 
\\[1.5mm] 
&\tau\varphi_{\tau} \to 0 
\quad \mbox{strongly in}\ H^1(0, T; H),  \label{Ptau.conv3} 
\\[1.5mm] 
&\theta_{\tau} \to \theta  
\quad \mbox{weakly in}\ L^2(0, T; V) \subset L^2(0, T; L^2(\Gamma)), 
\label{Ptau.conv4} 
\\[1.5mm] 
&c_{s}\ln \theta_{\tau} + \lambda(\varphi_{\tau}) \to z 
\quad \mbox{weakly$^{*}$ in}\ H^1(0, T; V^{*}) \cap L^{\infty}(0, T; H), 
\label{Ptau.conv5} 
\\[1.5mm] 
&c_{s}\ln \theta_{\tau} + \lambda(\varphi_{\tau}) \to z 
\quad \mbox{strongly in}\ C([0, T]; V^{*}), 
\label{Ptau.conv6} 
\\[1.5mm] 
&\mu_{\tau} \to \mu \quad \mbox{weakly in}\ L^2(0, T; V), 
\label{Ptau.conv7} 
\\[1.5mm] 
&\xi_{\tau} \to \xi  
\quad \mbox{weakly in}\ L^2(0, T; H) \label{Ptau.conv8}
\end{align}
as $\tau = \tau_{j} \searrow 0$. 
The Taylor formula with integral remainder implies that 
\begin{align*}
|\lambda(\varphi_{\tau})-\lambda(\varphi)| 
&\leq |\lambda'(\varphi)||\varphi_{\tau}-\varphi|  
         + \frac{\|\lambda''\|_{L^{\infty}(\mathbb{R})}}{2}|\varphi_{\tau}-\varphi|^2 
\\ \notag 
&\leq  \|\lambda''\|_{L^{\infty}(\mathbb{R})}|\varphi||\varphi_{\tau}-\varphi| 
         + |\lambda'(0)||\varphi_{\tau}-\varphi| 
         + \frac{\|\lambda''\|_{L^{\infty}(\mathbb{R})}}{2}|\varphi_{\tau}-\varphi|^2.  
\end{align*}
Thus there exists a constant $C_{1} > 0$ such that 
\begin{align*}
\|\lambda(\varphi_{\tau})-\lambda(\varphi)\|_{L^2(0, T; H)}  
&\leq C_{1}\|\varphi\|_{L^{\infty}(0, T; V)}
                          \|\varphi_{\tau}-\varphi\|_{L^2(0, T; V)}  
         + C_{1}\|\varphi_{\tau}-\varphi\|_{L^2(0, T; H)} 
\\ \notag   
    &\,\quad+ C_{1}\|\varphi_{\tau}-\varphi\|_{L^{\infty}(0, T; V)}
                                                 \|\varphi_{\tau}-\varphi\|_{L^2(0, T; V)} 
\end{align*} 
for all $\tau \in (0, \overline{\tau})$,   
whence it holds that   
\begin{align}\label{conv.lambda}
\lambda(\varphi_{\tau}) \to \lambda(\varphi) 
\quad \mbox{strongly in}\ L^2(0, T; H)
\end{align} 
as $\tau = \tau_{j} \searrow 0$. 
Then we deduce 
from \eqref{Ptau.conv4}, \eqref{Ptau.conv6} and \eqref{conv.lambda} that 
\begin{align*}
&\int_{0}^{T} \bigl(\ln \theta_{\tau}(t), \theta_{\tau}(t) \bigr)_{H}\,dt 
\\ \notag 
&= \frac{1}{c_{s}}\int_{0}^{T} 
                        \langle 
                          c_{s}\ln \theta_{\tau}(t) + \lambda(\varphi_{\tau}(t)), 
                                                                                 \theta_{\tau}(t) 
                  \rangle_{V^{*}, V}\,dt 
     - \frac{1}{c_{s}}\int_{0}^{T} 
                           \bigl(\lambda(\varphi_{\tau}(t)), \theta_{\tau}(t) \bigr)_{H}\,dt 
\\ \notag 
&\to \frac{1}{c_{s}}\int_{0}^{T} 
                             \langle z(t), \theta(t) \rangle_{V^{*}, V}\,dt 
     - \frac{1}{c_{s}}\int_{0}^{T} 
                           \bigl(\lambda(\varphi(t)), \theta(t) \bigr)_{H}\,dt 
\\ \notag 
&= \frac{1}{c_{s}}\int_{0}^{T} 
                           \bigl(z(t)-\lambda(\varphi(t)), \theta(t) \bigr)_{H}\,dt 
\end{align*} 
as $\tau = \tau_{j} \searrow 0$. 
Thus we have that $\frac{1}{c_{s}}(z-\lambda(\varphi))=\ln \theta$ 
and consequently  
\begin{align}\label{P.ln}
z = c_{s}\ln \theta + \lambda(\varphi)   
\end{align} 
a.e.\ on $\Omega\times(0, T)$ (see, e.g., \cite[Lemma 1.3, p.\ 42]{Barbu1}). 
On the other hand, we infer from \eqref{Ptau.conv2} and \eqref{Ptau.conv8} that 
\begin{align*}
\int_{0}^{T} (\xi_{\tau}(t), \varphi_{\tau}(t))_{H}\,dt 
\to \int_{0}^{T} (\xi(t), \varphi(t))_{H}\,dt
\end{align*}
as $\tau = \tau_{j} \searrow 0$,  
which yields that 
\begin{align}\label{P.beta}
\xi \in \beta(\varphi)  
\end{align}
a.e.\ on $\Omega\times(0, T)$ (see, e.g., \cite[Lemma 1.3, p.\ 42]{Barbu1}). 

Therefore we can verify that the above functions $\varphi$, $\theta$, $\mu$ 
and $\xi$ satisfy \eqref{dfPsol1}-\eqref{dfPsol4}
by 
\eqref{dfPtausol1}-\eqref{dfPtausol4},  
\eqref{Ptau.conv1}-\eqref{Ptau.conv8}, \eqref{P.ln} and \eqref{P.beta}, 
which means that Theorem \ref{maintheorem2} holds. 
\qed
\end{prth2.2}

\vspace{10pt}

\section*{Acknowledgments}
The research of PC is supported by the Italian Ministry of Education, 
University and Research~(MIUR): Dipartimenti di Eccellenza Program (2018--2022) 
-- Dept.~of Mathematics ``F.~Casorati'', University of Pavia. 
In addition, PC gratefully acknowledges some other 
support from the MIUR-PRIN Grant 2015PA5MP7 
``Calculus of Variations'' 
and the GNAMPA (Gruppo Nazionale per l'Analisi Matematica, 
la Probabilit\`a e le loro Applicazioni) of INdAM (Isti\-tuto 
Nazionale di Alta Matematica). 
The research of SK is supported by JSPS Research Fellowships 
for Young Scientists (No.\ 18J21006) 
and JSPS Overseas Challenge Program for Young Researchers. 
%
%
%
%%==============================================================%%
%%==============                                  ==============%%
%%======                                                  ======%%
%%====                                                      ====%%
%%==                         Reference                        ==%%
%%====                                                      ====%%
%%======                                                  ======%%
%%==============                                  ==============%%
%%==============================================================%%


\begin{thebibliography}{99}

\bibitem{Barbu1}
V. Barbu,
``Nonlinear semigroups and differential equations in Banach spaces'',
Noord\-hoff,
Leyden,
1976. 

\bibitem{Barbu2}
V. Barbu, 
``Nonlinear differential equations of monotone types in Banach spaces'',
Springer Monographs in Mathematics. Springer, New York, 2010. 

\bibitem{Bon2005}
E. Bonetti, 
{\it A new approach to phase transitions with thermal memory via 
the entropy balance}, 
Mathematical methods and models in phase transitions, 
125--155, Nova Sci. Publ., New York,  2005. 

\bibitem{BBR}  
E. Bonetti, G. Bonfanti, R. Rossi,  
{\it Analysis of a temperature-dependent model for adhesive contact with friction}, 
Phys. D {\bf 285} (2014), 42--62.  

\bibitem{bcfgdue}
E. Bonetti, P. Colli, M. Fabrizio, G. Gilardi, 
{\it Modelling and long-time behaviour for phase transitions 
with entropy balance and thermal memory conductivity}, 
Discrete Contin. Dyn. Syst. Ser. B {\bf 6} (2006), 1001--1026.

\bibitem{BCFG2007} 
E. Bonetti, P. Colli, M. Fabrizio, G. Gilardi,  
{\it Global solution to a singular integro-differential system 
related to the entropy balance}, 
Nonlinear Anal. {\bf 66} (2007), 1949--1979.

\bibitem{BCFG2009}  
E. Bonetti, P. Colli, M. Fabrizio, G. Gilardi,  
{\it Existence and boundedness of solutions for a singular phase field system}, 
J. Differential Equations {\bf 246}  (2009), 3260--3295.

\bibitem{BCF} 
E. Bonetti, P. Colli, M. Fr\'emond:
{\it A phase field model with thermal memory governed by the entropy balance\/},
Math. Models Methods Appl. Sci.
{\bf 13} (2003), 1565--1588. 

\bibitem{BCG} 
E. Bonetti, P. Colli, G. Gilardi,  
{\it Singular limit of an integrodifferential system related to the entropy balance}, 
Discrete Contin. Dyn. Syst. Ser. B  {\bf 19}  (2014), 1935--1953.

\bibitem{BF}
E. Bonetti, M. Fr\'emond, 
{\it A phase transition model with the entropy balance\/},
Math. Methods Appl. Sci.
{\bf 26} (2003), 539--556. 

\bibitem{BFR}
E. Bonetti, M. Fr\'emond, E. Rocca,  
{\it A new dual approach for a class of phase transitions with memory: 
existence and long-time behaviour of solutions}, 
J. Math. Pures Appl. (9) {\bf 88} (2007), 455--481.	

\pier{\bibitem{BR2007}
E. Bonetti, E. Rocca, 
{\it Global existence and long-time behaviour for 
a singular integro-differential phase-field system}, 
Commun. Pure Appl. Anal. {\bf 6} (2007), 367--387.}

\bibitem{Brezis} 
H. Brezis, 
``Op\'erateurs maximaux monotones et \pier{semi-groupes} de contractions
dans les espaces de Hilbert'',
North-Holland Math. Stud.
{\bf 5},
North-Holland,
Amsterdam,
1973. 

\pier{
\bibitem{CC1}
G. Canevari, P. Colli, 
{\it Solvability and asymptotic analysis
of a generalization of the Caginalp phase field system}, 
Commun. Pure Appl. Anal. {\bf 11} (2012), 1959--1982.
%
\bibitem{CC2}
G. Canevari, P. Colli,  
{\it Convergence properties
for a generalization of the Caginalp phase field system}, 
Asymptot. Anal. {\bf 82} (2013), 139--162.}

\bibitem{CC}
P. Colli, M. Colturato,   
{\it Global existence for a singular phase field system related to 
a sliding mode control problem}, 
Nonlinear Anal. Real World Appl. {\bf 41} (2018), 128--151.

\bibitem{CGGS1}
P. Colli, G. Gilardi, M. Grasselli, G. Schimperna,   
{\it The conserved phase-field system with memory}, 
Adv. Math. Sci. Appl. {\bf 11} (2001), 265--291. 

\bibitem{CGGS2}
P. Colli, G. Gilardi, M. Grasselli, G. Schimperna,  
{\it Global existence for the conserved phase field model 
with memory and quadratic nonlinearity}, 
Port. Math. (N.S.) {\bf 58} (2001), 159--170.

\bibitem{CGLN}
P. Colli, G. Gilardi, P. Lauren\c{c}ot, A. Novick-Cohen,  
{\it Uniqueness and long-time behavior for the conserved phase-field system 
with memory}, 
Discrete Contin. Dynam. Systems {\bf 5} (1999), 375--390. 

\bibitem{CK1}
P. Colli, S.\ Kurima, 
    {\it Time discretization of a nonlinear phase field system 
in general domains}, \pier{preprint arXiv:1811.10730 [math.AP] (2018), pp.~1-24.}

\bibitem{C2018}
M. Colturato, 
{\it Well-posedness and longtime behavior for a singular phase field system 
with perturbed phase dynamics}, 
Evol. Equ. Control Theory {\bf 7} (2018), 217--245.
		
\pier{\bibitem{CGMQ}
M. Conti, S. Gatti, A. Miranville, R. Quintanilla,  
{\it On a Caginalp phase-field system with two temperatures and memory}, 
Milan J. Math.  {\bf 85} (2017), 1--27.}

\bibitem{FaGioMo} 
M. Fabrizio, C. Giorgi,  A. Morro,
{\it Internal dissipation, relaxation properties, and free energy 
in materials with fading memory\/},
J.~Elasticity
{\bf 40} (1995), 107--122.

\bibitem{Frel}
M. Fr\'emond,
``Non-smooth thermomechanics'',
Springer-Verlag, Berlin, 2002.

\bibitem{GMS}
G. Gilardi, A. Miranville, G. Schimperna, 
{\it On the Cahn--Hilliard equation with irregular potentials 
and dynamic boundary conditions},  
Commun. Pure Appl. Anal. {\bf 8} (2009), 881--912.

\bibitem{GR2006}
G. Gilardi, E. Rocca, 
{\it Convergence of phase field to phase relaxation models governed 
by an entropy equation with memory}, 
Math. Methods Appl. Sci. {\bf 29} (2006), 2149--2179.

\bibitem{GR2007} 
G. Gilardi, E. Rocca, 
{\it Well-posedness and long-time behaviour for a singular phase field system 
of conserved type}, 
IMA J. Appl. Math. {\bf 72} (2007), 498--530.

\pier{\bibitem{GreenNaghdi}
A.E. Green, P.M. Naghdi, 
{\it A re-examination of the basic postulates of thermo-mechanics}, 
Proc. Roy. Soc. Lond. A {\bf 432} (1991), 171--194.}

\bibitem{Gu}
M.E. Gurtin,
{\it Generalized Ginzburg--Landau and Cahn--Hilliard equations based on a
microforce balance\/},
Phys.~D {\bf 92} (1996), 178--192. 

\pier{%
\bibitem{GPmodel}
 M.E. Gurtin, A.C. Pipkin,
{\it A general theory of heat conduction with finite wave speeds},
Arch. Rational Mech. Anal. {\bf 31} (1968), 113--126}

\bibitem{Jerome}
J.W. Jerome, 
``Approximations of Nonlinear Evolution Systems'',
Mathematics in Science and Engineering {\bf 164},
Academic Press Inc., Orlando, 1983. 

\bibitem{KS}
J. Kou, S. Sun,  
{\it Thermodynamically consistent modeling and simulation of 
multi-component two-phase flow with partial miscibility}, 
Comput. Methods Appl. Mech. Engrg. {\bf 331} (2018), 623--649.
	
\bibitem{M2017}
A. Montanaro,  
{\it On thermo-electro-mechanical simple materials with fading memory: 
restrictions of the constitutive equations in a Green-Naghdi type theory}, 
Meccanica {\bf 52}  (2017), 3023--3031.
		  
\bibitem{MiSc}
A. Miranville, G. Schimperna,
{\it Nonisothermal phase separation based on a microforce balance\/},
Discrete Contin. Dyn. Syst. Ser. B  {\bf 5} (2005), 753--768.

\bibitem{Muller}
I. M\" uller,
{\it Thermodynamics of mixtures and phase field theory\/},
Int. J. Solids Struct. {\bf 38} (2001), 1105--1113.

\bibitem{N1967}
J. Ne\v{c}as, 
``Les m\'ethodes directes en th\'eorie des \'equations elliptiques'', 
(French) Masson et Cie (Eds.), Paris, Academia, Editeurs, Prague, 1967. 

\pier{\bibitem{Podio1}
P. Podio-Guidugli, 
{\it A virtual power format for thermomechanics}, 
Contin. Mech. Thermodyn. {\bf 20} (2009), 479--487.}

\end{thebibliography}
\end{document}